\documentclass{amsart}
\usepackage[utf8]{inputenc}
\usepackage[square,numbers]{natbib}
\usepackage[hidelinks]{hyperref}
\usepackage{amsthm}
\usepackage{amssymb}
\usepackage{wasysym}
\usepackage[square,numbers]{natbib}
\usepackage{enumitem}
\usepackage{cleveref}
\usepackage{color}
\usepackage{tikz}
\usepackage{pgfplots}
\usepackage{enumitem}
\usepackage{tabularx}
\usepackage{etoolbox}
\usepackage{ifthen}
\usepackage{caption}

\author{Nick Bezhanishvili}
\author{Antonio Maria Cleani}
\title{Blok-Esakia Theorems via Stable  Canonical Rules}
\address{Institute for Logic, Language and Computation, University of Amsterdam}
\address{Department of Philosophy, University of Southern California}

\theoremstyle{plain}
\newtheorem{theorem}{Theorem}[section]
\newtheorem{proposition}[theorem]{Proposition}
\newtheorem{lemma}[theorem]{Lemma}
\newtheorem{corollary}[theorem]{Corollary}

\theoremstyle{definition}
\newtheorem{definition}[theorem]{Definition}

\newtheorem{remark}[theorem]{Remark}
\newtheorem{convention}[theorem]{Convention}

\AfterEndEnvironment{definition}{\noindent\ignorespaces}
\AfterEndEnvironment{proposition}{\noindent\ignorespaces}
\AfterEndEnvironment{theorem}{\noindent\ignorespaces}
\AfterEndEnvironment{lemma}{\noindent\ignorespaces}
\AfterEndEnvironment{example}{\noindent\ignorespaces}
\AfterEndEnvironment{remark}{\noindent\ignorespaces}
\AfterEndEnvironment{claim}{\noindent\ignorespaces}
\AfterEndEnvironment{corollary}{\noindent\ignorespaces}

\newcommand{\z}{Zakharyaschev}
\newcommand{\pastlog}{\blacklozenge}
\newcommand{\pastsquare}{\blacksquare}
\newcommand{\sisquare}{\boxtimes}

\newcommand{\logic}[1]{\mathtt{#1}}
\newcommand{\alg}[1]{\mathfrak{#1}}
\newcommand{\spa}[1]{\mathfrak{#1}}
\newcommand{\class}[1]{\mathcal{#1}}
\newcommand{\op}[1]{\mathsf{#1}}
\newcommand{\lat}[1]{\mathbf{#1}}

\newcommand{\chunk}{\mathit{ch}}

\newcommand{\dualspa}[1]{\alg{#1}_*}
\newcommand{\dualalg}[1]{\spa{#1}^*}

\newcommand{\dualkr}[1]{\alg{#1}_+}
\newcommand{\dualcalg}[1]{\spa{#1}^+}

\newcommand{\ops}[1]{\mathit{op}(#1)}

\newcommand{\clop}[2][]{\op{Clop}_{#1}(\spa{#2})}
\newcommand{\clopup}[2][]{\op{ClopUp}_{#1}(\spa{#2})}

\newcommand{\upset}[2][]{{\Uparrow_{#1}} #2}
\newcommand{\downset}[2][]{{\Downarrow_{#1}} #2}

\makeatletter
\providecommand{\leftsquigarrow}{%
	\mathrel{\mathpalette\reflect@squig\relax}%
}
\newcommand{\reflect@squig}[2]{%
	\reflectbox{$\m@th#1\rightsquigarrow$}%
}
\makeatother

\newcommand{\scrmod}[2]{\mu(\alg{#1}, {#2})}

\newcommand{\scr}[2]{{\xi}(\alg{#1}, {#2})}

\newcommand{\scrsi}[2]{\eta(\alg{#1}, {#2})}

\newcommand{\ruletrans}[2]{\mu_\circ(\alg{#1}, {#2})}

\newcommand{\from}{\leftarrow}
\newcommand{\gen}{\rotatebox[origin=c]{180}{\ensuremath\neg}}

\newcommand{\sig}{s}
\newcommand{\simod}{1}
\newcommand{\bsiten}{2}
\newcommand{\msimod}{3}

\newcommand{\least}[1][]{\tau_{#1}}
\newcommand{\greatest}[1][]{\sigma_{#1}}
\newcommand{\fragment}[1][]{\rho_{#1}}

\newcommand{\transl}[1][]{T_{#1}}
\usetikzlibrary{positioning}
\usetikzlibrary{calc}
\usetikzlibrary{arrows.meta,arrows}
\tikzset{
	modal/.style={>=stealth,shorten >=1pt,shorten <=1pt,auto,node distance=1.5cm,
		semithick},
	world/.style={rectangle,draw,minimum size=0.8cm,fill=gray!15},
	point/.style={circle,draw,inner sep=0.5mm,fill=black},
	reflexive above/.style={->,loop,looseness=5,in=120,out=60},
	reflexive below/.style={->,loop,looseness=5,in=240,out=300},
	reflexive left/.style={->,loop,looseness=5,in=150,out=210},
	reflexive right/.style={->,loop,looseness=5,in=30,out=330}
}

\tikzset{roundnode/.style={draw,circle,inner sep=1pt,outer sep=0pt}
}
\tikzset{irr/.style={draw,circle,fill,inner sep=1pt,outer sep=0pt}
}

\tikzset{>={Computer Modern Rightarrow[width=2.5pt,length=2pt]}}
\tikzset{label distance=-3pt}

\pgfplotsset{compat=1.17}

\begin{document}
\maketitle

	\begin{abstract}
	We present a new uniform method for studying modal companions of superintuitionistic rule systems and related notions, based on the machinery of stable canonical rules. Using this method, we obtain  alternative proofs of the Blok-Esakia theorem and of the Dummett-Lemmon conjecture for rule systems. Since stable canonical rules may be developed for any rule system admitting filtration, our method generalizes smoothly to richer signatures. Using essentially the same argument, we obtain a proof of an analogue of the Blok-Esakia theorem for bi-superintuitionistic and tense rule systems, and of the Kuznetsov-Muravitsky isomorphism between rule systems extending the modal intuitionistic logic $\logic{KM}$ and modal rule systems extending the provability logic  $\logic{GL}$. In addition, our proof of the Dummett-Lemmon conjecture also generalizes to the bi-superintuitionistic and tense cases. 
\end{abstract}

	\section*{Introduction}

 A modal companion of a superintuitionistic logic $\mathtt{L}$ is defined as any normal modal logic $\mathtt{M}$ extending  $\mathtt{S4}$ such that the \emph{Gödel translation} fully and faithfully embeds  $\mathtt{L}$ into $\mathtt{M}$. The notion of a modal companion has sparked a remarkably prolific line of research, documented, e.g., in the surveys \cite{ChagrovZakharyashchev1992MCoIPL} and \cite{WolterZakharyaschev2014OtBET}. The jewel of this research line is the celebrated \emph{Blok-Esakia theorem}, first proved independently by \citet{Blok1976VoIA} and \citet{Esakia1976OMCoSL}. The theorem states that the lattice of superintuitionistic logics is isomorphic to the lattice of normal extensions of Grzegorczyk's modal logic $\mathtt{Grz}$, via the mapping which sends each superintuitionistic logic $\mathtt{L}$ to the normal extension of $\mathtt{Grz}$ with the set of all Gödel translations of formulae in $\mathtt{L}$. 

\citet{Zakharyashchev1991MCoSLSSaPT} developed a unified approach to the theory of modal companions, via his technique of \emph{canonical formulae}. These formulae generalize the subframe formulae of \citet{Fine1985LCKPI}. Like a subframe formula, a canonical formula syntactically encodes the structure of a finite \emph{refutation pattern}, i.e., a finite transitive frame together with a (possibly empty) set of parameters. By applying a version of the \emph{selective filtration} construction, every formula can be matched with a finite set of finite refutation patterns, in such a way that the conjunction of all the canonical formulae associated with the refutation patterns is equivalent to the original formula. By studying how the Gödel translation affects superintuitionistic canonical formulae, 
Zakharyashchev gave alternative proofs of classic theorems in the theory of modal companions, and extended this theory with several novel results.  Among these, he confirmed the \emph{Dummett-Lemmon conjecture}, formulated in \cite{DummettLemmon1959MLbS4aS5}, which states that a superintuitionistic logic is Kripke complete iff its weakest modal companion is. \citet{Jerabek2009CR} generalized canonical formulae to \emph{canonical rules}, and applied this notion to extend \z's approach to theory of modal companions to \emph{rule systems} (also known as \emph{multi-conclusion consequence relations.})

In \cite{BezhanishviliEtAl2016SCR,BezhanishviliEtAl2016CSL,BezhanishviliBezhanishvili2017LFRoHAaCF}, \emph{stable canonical formulae} and \emph{rules} were introduced as an alternative to Zakharyaschev and Je\v{r}ábek-style canonical rules and formulae. The basic idea is the same: a stable canonical formula or rule syntactically encodes the semantic structure of a finite refutation pattern. The main difference lies in how such structure is encoded, which affects how refutation patterns are constructed in the process of rewriting a formula (or rule) into a conjunction of stable canonical formulae (or rules). Namely, in the case of stable canonical formulae and rules finite refutation patterns are constructed by taking \emph{filtrations} rather than selective filtrations of countermodels. A survey of stable canonical formulae and rules can be found in \cite{Bezhanishvili2023JFaATfIL}.

In this paper, we apply stable canonical rules to develop a novel, uniform approach to the study of modal companions and similar notions in richer signatures. Our approach echoes the \z-Je\v{r}ábek approach in using rules encoding finite refutation patterns, but also bears circumscribed similarities with Blok's original algebraic approach in some proof strategies (see \Cref{remark:blok}). Our techniques deliver central results in the theory of modal companions through transparent arguments. In particular, we obtain an alternative proof of the Blok-Esakia theorem for both logics and rule systems, and generalize the Dummett-Lemmon conjecture to rule systems. 

Due to the flexibility of filtration, our techniques generalize smoothly to rule systems in richer signatures. To illustrate this, we apply our method to the study of \emph{tense companions} of bi-superintuitionistic deductive systems and to the study of (mono)modal companions of modal intuitionistic rule systems above $\logic{KM}$. In each of these cases, we prove analogues of the Blok-Esakia theorem. When restricted to logics, these results were proved, respectively, by \citet[Theorem 23]{Wolter1998OLwC} and \cite[Proposition 3]{KuznetsovMuravitsky1986OSLAFoPLE}, though they appear to be new for rule systems. In the case of tense companions, in addition, we also prove an analogue of the Dummett-Lemmon conjecture for rule systems, which also appears to be novel. 

Notably, in each of these three cases, our main results are proved by essentially the same arguments. By contrast, generalizing the \z-Je\v{r}ábek technique beyond the case of modal companions of superintuitionistic logics is far from straightforward. In particular, as we argue towards the end of \Cref{sec:scr}, it is far from clear whether  the \z-Je\v{r}ábek technique generalizes to the case of tense companions, since selective filtration does not work well for bi-superintuitionistic and tense logics.

The techniques described in this paper can also be used to obtain axiomatic characterizations of the modal companion maps (and their counterparts in the richer signatures discussed here) in terms of stable canonical rules, as well as some results concerning the notion of stability \cite{BezhanishviliEtAl2018SML}. These results can be found in the recent master's thesis \cite{Cleani2021TEVSCR}, on which the present paper is based.

The paper is organized as follows. We begin by reviewing some general preliminaries in \Cref{sec:preliminaries}, followed by the basic constructions in the theory of modal companions in \Cref{sec:mappings}. We then introduce stable canonical rules in \Cref{sec:scr}, generalizing known constructions to the bi-superintuitionistic and tense case. In \Cref{sec:mc} we present our proof of a general Blok-Esakia theorem, which uniformly applies to each of the three notions of companions we are interested in. Finally, in \Cref{sec:dummettlemmon} we present our proof of a general Dummett-Lemmon conjecture, applying to both modal and tense companions. We conclude in \Cref{sec:conclusion}.

	\section{Preliminaries}

        \label{sec:preliminaries}

	We review some basic facts about rule systems and their interpretation over algebras, topological spaces and Kripke frames. The reader may consult the following references for more detailed information:  \cite{Iemhoff2016CRaAR} for rule systems in general; \cite{ChagrovZakharyaschev1997ML,Jerabek2009CR} for modal and superintuitionistic rule systems;  \cite{BurrisSankappanavar1981ACiUA} for universal algebra; \cite{Johnstone1982SS,Esakia2019HADT} for duality theory. 

	\subsection{Rule Systems}

	Throughout the paper we fix a countably infinite set of propositional variables $\mathit{Prop}$. For a signature $\nu$ (a finite set of propositional connectives), the set $\mathit{Frm}_\nu$ of \emph{$\nu$-formulae} is built from $\mathit{Prop}$ using the connectives in $\nu$ in the usual way. A \emph{substitution} is a map $s:\mathit{Frm}_\nu(\mathit{Prop})\to \mathit{Frm}_\nu(\mathit{Prop})$ which commutes with the operators in $\nu$. 

	A \emph{rule} in signature $\nu$ is a pair  $(\Gamma, \Delta)$ such that $\Gamma, \Delta$ are finite subsets of $\mathit{Frm}_\nu$. In case $\Delta=\{\varphi\}$ we write $\Gamma/\Delta$ simply as $\Gamma/\varphi$, and analogously if $\Gamma=\{\psi\}$. Moreover, we write $/\varphi$ for the rule $\varnothing/\varphi$. A rule is said to be \emph{single-conclusion} when of the form $\Gamma/\varphi$, and \emph{assumption free} when of the form $/\Delta$.  We use $;$ to denote union between finite sets of formulae, so that $\Gamma; \Delta=\Gamma\cup \Delta$ and $\Gamma; \varphi=\Gamma\cup \{\varphi\}$. We let $\mathit{Rul}_\nu$ be the set of all rules in $\nu$. 

	\begin{definition}
		A \emph{rule system}\footnote{Rule systems are also called \emph{multiple-conclusion consequence relations} (e.g., in \cite{BezhanishviliEtAl2016SCR,Iemhoff2016CRaAR}). We prefer the terminology of rule systems (used in \cite{Jerabek2009CR}) for brevity.} in signature $\nu$  is a set $\logic{S}\subseteq \mathit{Rul}_\nu$ satisfying the following conditions:
		\begin{enumerate}
			\item If $\Gamma/\Delta\in \logic{S}$, then $ s[\Gamma]/ s[\Delta]\in \logic{S}$ for all substitutions $s$ (structurality);
			\item $\varphi/\varphi\in \logic{S}$ for every formula $\varphi$ (reflexivity);
			\item If $\Gamma/\Delta\in \logic{S}$, then $\Gamma;\Gamma'/\Delta;\Delta'\in \logic{S}$ for any finite sets of formulae $\Gamma',\Delta'$ (monotonicity);
			\item If $\Gamma/\Delta;\varphi\in\logic{S}$ and $\Gamma;\varphi/\Delta\in \logic{S}$, then $\Gamma/\Delta\in \logic{S}$ (cut).
		\end{enumerate}
	\end{definition}
	If $\class{S}$ is a set of rule systems and  $\Sigma, \Xi$ are sets of rules, we write $\Xi\oplus_{\class{S}}\Sigma$ for the least rule system in $\class{S}$, if it exists, extending both $\Xi$ and $\Sigma$. A set of rules $\Sigma$ is said to \emph{axiomatize} a rule system $\logic{S}\in \class{S}$ \emph{over} some rule system $\logic{S}'\in \class{S}$ if $\logic{S}'\oplus_{\class{S}}\Sigma=\logic{S}$. Normally, we will take $\class{S}$ to be the set of all extensions of some particular rule system. When $\class{S}$ is clear from context, we write simply $\oplus$ instead of $\oplus_{\class{S}}$. 

	In  this paper we will work with rule systems in 5 different signatures. 
	\begin{itemize}
		\item The \emph{modal} signature $m:=\{\land, \neg, \bot, \square\}$;
		\item The \emph{tense} signature $t:=\{\land, \neg, \bot, \square, \pastlog\}$;
		\item The \emph{superintuitionistic} (si) signature $si:=\{\land, \lor, \to, \bot, \top\}$;
		\item The \emph{bi-superintuitionistic} (bsi) signature $bsi:=\{\land, \lor, \to, \from, \bot, \top\}$; 
		\item The \emph{modal superintuitionistic} (msi) signature $msi:=\{\land, \lor, \to, \sisquare, \top, \bot\}$.
	\end{itemize}
	When working in the modal and tense signatures, we will treat the other Boolean and modal connectives as defined in the usual way. We will denote the duals of $\square$ and $\pastlog$ as $\lozenge$ and $\pastsquare$ respectively. In the bsi signature we also use the abbreviations
	\[\neg p:= p\to \bot,\qquad \gen p:=\top\from p.\]

	For each unary propositional connective $\heartsuit$, we define the rules
	\begin{align}
		&/\heartsuit (p\to q)\to (\heartsuit p\to \heartsuit q), \tag{K$_\heartsuit$}\\
		&\varphi/\heartsuit\varphi. \tag{Nec$_\heartsuit$}
	\end{align}
	A \emph{normal modal rule system} is a rule system in the signature $m$ containing the rule $/\varphi$ whenever $\varphi$ is a theorem of the Classical Propositional Calculus, as well as the rules (K$_\square$), (Nec$_\square$) and
	\[\varphi\to \psi, \varphi/\psi. \label{mp}\tag{MP}\]
	A \emph{normal tense rule system} is a rule system in the signature $t$, whose $\square$-free and $\pastlog$-free fragments are each a normal modal rule system (with respect to $\square$ and $\pastlog$ respectively) and which, in addition, contains the rule
	\[/\varphi\to \square \pastlog \varphi.\label{t}\tag{t}\]
	We will henceforth omit the prefix ``normal.''

	A \emph{si rule system} is a rule system in the signature $si$ containing the rule $/\varphi$ whenever $\varphi$ is a theorem of the intuitionistic propositional calculus $\logic{IPC}$, as well as the rule (\ref{mp}). A \emph{bsi rule system} is a rule system in the signature $bsi$ containing the rule $/\varphi$ whenever $\varphi$ is a theorem of the bi-intutionistic propositional calculus $\logic{biIPC}$, as well as the rules (\ref{mp}) and (Nec$_{\neg\gen}$). We refer the reader to \cite[Ch. 12]{ChagrovZakharyaschev1997ML} and \cite{Rauszer1974AFotPCoHBL} for explicit axiomatizations of $\logic{IPC}$ and $\logic{biIPC}$ respectively. 
	
	Finally, a \emph{msi rule system} is a rule system in the signature $msi$, 
	whose $si$ fragment is a si rule system and which, in addition, contains the rules (K$_\sisquare$) and (Nec$_\sisquare$), as well as the following:
	\begin{gather}
		/p\to \sisquare p\\
		{/}{\sisquare} p \to (q \lor (q\to p)).
	\end{gather}

	When $\logic{M}$ is a modal (resp.\ tense, msi) rule system, we write $\lat{NExt}(\logic{M})$ for the class of all modal (resp.\ tense, msi) rule systems extending $\logic{M}$. Similarly, when $\logic{L}$ is a si or bsi rule system, we write $\lat{Ext}(\logic{L})$ for the class of all si or bsi rule systems extending $\logic{L}$. We note that all these classes of rule systems form complete lattices, where the meet is intersection and the join is given by the $\oplus$ operation over the relevant class of rule systems.

	A \emph{(modal, tense, si, bsi or msi) logic} is a (modal, tense, si, bsi or msi) rule system which can be axiomatized, over the least rule system of the same kind, by a set of assumption-free, single conclusion rules.  Logics in this sense correspond one-to-one with logics conceived of as sets of formulae closed under appropriate conditions, a conception that much of the literature in the field of modal and superintuitionitstic logic shares. For example, the (normal) modal logics in the standard sense \cite[e.g.,][p. 113]{ChagrovZakharyaschev1997ML} correspond one-to-one with the normal modal rule systems axiomatizable by assumption-free, single conclusion rules. When $\logic{M}$ is a modal logic in this sense, there is always a corresponding  modal rule system $/\logic{M}$ axiomatized by $\{/\varphi:\varphi\in \logic{M}\}$. Conversely,  for any modal rule system $\logic{N}$ the set $\{\varphi:/\varphi\in \logic{N}\}$ is always a modal logic in the standard sense.

 \begin{convention}
     In view of this correspondence, we will use familiar names for standard logics in the literature to refer to the corresponding rule system: that is, when $\logic{S}$ names a standard (modal, tense, si, bsi or msi) logic we shall identify $\logic{S}$ with the rule system $/\logic{S}$ defined as above. Thus, for example, we write $\logic{K}$ for the least modal rule system,  $\logic{IPC}$ for the least si rule system, and so on. 
 \end{convention}

    The set $\class{L}$ of rule systems of a given kind which admit an assumption-free, single conclusion axiomatization forms a complete lattice. Moving forward, When $\lat{NExt}(\logic{M})$ (resp.\ $\lat{Ext}(\logic{L}))$ is a lattice of rule systems, we denote the corresponding lattice of logics as $\lat{NExtL}(\logic{M})$ (resp.\ $\lat{ExtL}(\logic{L}))$. 
    The join $\oplus_{\lat{NExtL}(\logic{M})}$ coincides with $\oplus_{\lat{NExt}(\logic{M})}$. However, the meets $\otimes_{\lat{NExtL}(\logic{M})}$ and $\otimes_{\lat{NExt}(\logic{M})}$ generally come apart: the meet in $\lat{NExt}(\logic{M})$ of two logics may itself fail to be a logic. Likewise in the si and bsi cases. 
    
    The meet $\otimes_{\lat{NExtL}(\logic{M})}$ can be characterized in terms of $\otimes_{\lat{NExt}(\logic{M})}$ as follows. If $\logic{N}\in \lat{NExt}(\logic{M})$ is a rule system, let $\op{Taut}(\logic{N})$ be the logic axiomatized over $\logic{M}$ by all assumption-free, single conclusion belonging to $\logic{N}$. Analogously, we define $\op{Taut}(\logic{L})$  when $\logic{L}$ is a si or bsi rule system.
    \begin{proposition}
        Let $\class{S}:=\lat{NExt}(\logic{M})$ $($resp.~$\lat{Ext}(\logic{L}))$ and let $\class{L}:=\lat{NExtL}(\logic{M})$ $($resp.\ $\lat{ExtL}(\logic{L}))$. Then the identity
        \[\otimes_{\class{L}}\{\logic{S}_i:i\in I\}=\op{Taut}\left(\otimes_{\class{S}}\{\logic{S}_i:i\in I\}\right). \]
        holds for all logics $\{\logic{S}_i:i\in I\}\subseteq \class{L}$. \label{rulesystemlogic}
    \end{proposition}
    \begin{proof}
    We prove this result for binary meets; the proof of the general case is completely analogous. Note first that $\op{Taut}$ is monotonic and for each rule system $\logic{S}\in \class{S}$ we have $\op{Taut}(\logic{S})\subseteq \logic{S}$. Now, clearly $\logic{S}\otimes_{\class{L}}\logic{S'}\subseteq\logic{S}\otimes_{\class{S}} \logic{S'}.$  So,  $\op{Taut}(\logic{S}\otimes_{\class{L}}\logic{S'})\subseteq\op{Taut}(\logic{S}\otimes_{\class{S}} \logic{S'})$ by monotonicity. But $\op{Taut}(\logic{S}\otimes_{\class{L}}\logic{S'})=\logic{S}\otimes_{\class{L}}\logic{S'}$, so $\logic{S}\otimes_{\class{L}}\logic{S'}\subseteq\op{Taut}(\logic{S}\otimes_{\class{S}} \logic{S'})$. Conversely, by $\op{Taut}(\logic{S}\otimes_{\class{S}} \logic{S'})\subseteq \logic{S}\otimes_{\class{S}} \logic{S'}$ it follows that $\op{Taut}(\logic{S}\otimes_{\class{S}} \logic{S'})$ is a logic below both $\logic{S}$ and $\logic{S'}$, and so $\op{Taut}(\logic{S}\otimes_{\class{S}} \logic{S'})\subseteq \logic{S}\otimes_{\class{L}}\logic{S'}$.
    \end{proof}

	Throughout the paper we will refer to a number of standard rule systems. We collect all of them in \Cref{systems}. 
	\begin{table}
		
		\begin{tabular}[h]{|l|l|}
			\hline
			\multicolumn{2}{|c|}{Modal rule systems} \\
			\hline\hline
			$\logic{K}$ & The least modal rule system\\
			\hline
			$\logic{K4}$ & $\logic{K}\oplus /\square p\to \square\square p$\\
			\hline
			$\logic{S4}$ & $\logic{K4}\oplus /\square p\to p$\\
			\hline
			$\logic{Grz}$ & $\logic{S4} \oplus /\square (\square (p\to \square p)\to p)\to p$\\
			\hline
			$\logic{GL}$ & $\logic{K4}\oplus /\square (\square p\to p)\to \square p$\\
			\hline\hline
			\multicolumn{2}{|c|}{Si and bsi rule systems}\\
			\hline\hline
			$\logic{IPC}$ & The least si rule system\\
			$\logic{biIPC}$ & The least bsi rule system \\
			\hline\hline
			\multicolumn{2}{|c|}{Tense rule systems}\\
			\hline\hline
			$\logic{K.t}$ & The least tense rule system\\
			\hline
			$\logic{K4.t}$ & $\logic{K}\oplus /\square p\to \square\square p \oplus /\pastlog\pastlog p\to \pastlog p$\\
			\hline
			$\logic{S4.t}$ & $\logic{K4}\oplus /\square p\to p \oplus /p\to \pastlog p$\\
			\hline
			$\logic{Grz.t}$ & $\logic{S4.t}\oplus /\square (\square (p\to \square p)\to p)\to p \oplus /p\to \pastlog (p\land \neg \pastlog (\pastlog p\land \neg p))$\\
			\hline\hline
			\multicolumn{2}{|c|}{Msi rule systems}\\
			\hline\hline
			$\logic{IPCK}$ & The least msi rule system\\
			$\logic{KM}$ & $\logic{IPCK}\oplus  /(\sisquare p\to p)\to p$\\
			\hline
		\end{tabular}
		\caption{Standard rule systems}
		\label{systems}
	\end{table}

	\subsection{Algebraic Semantics}

	We interpret rule systems over algebras in the same signature. If $\alg{A}$ is a $\nu$-algebra, we denote its carrier as $A$. Let $\alg{A}$ be some $\nu$-algebra. A \emph{valuation} on $\alg{A}$ is a map $V:\mathit{Frm}_\nu\to A$, satisfying the condition 
\[V(f(\varphi_1, \ldots, \varphi_n)):=f^{\alg{A}}(V(\varphi_1), \ldots, V(\varphi_n))\]
for each $f\in \nu$. A pair $(\alg{A}, V)$ where $\alg{A}$ is a $\nu$-algebra and $V$ a valuation on $\alg{A}$ is called a \emph{model}. A model $(\alg{A}, V)$ \emph{satisfies} a rule $\Gamma/\Delta$ when the following holds: if $V(\gamma)=1$ for all $\gamma\in \Gamma$, then $V(\delta)=1$ for some $\delta\in \Delta$. In this case, we write $\alg{A}, V\models \Gamma/\Delta$. A rule $\Gamma/\Delta$ is \emph{valid} on a $\nu$-algebra $\alg{A}$ when $\alg{A}, V\models\Gamma/\Delta$ holds for all valuations $V$ on $\alg{A}$.  When this holds we write $\alg{A}\vDash \Gamma/\Delta$, otherwise we write $\alg{A}\nvDash \Gamma/\Delta$ and say that $\alg{A}$ \emph{refutes} $\Gamma/\Delta$. We can extend this notion of validity to classes of $\nu$-algebras in the obvious way.

Write $\class{A}_\nu$ for the class of all $\nu$-algebras. For every rule system $\logic{S}$ we define 
\[\op{Alg}(\logic{S}):=\{\alg{A}\in \class{A}_\nu:\alg{A}\vDash \logic{S}\}.\]
 Conversely, if $\class{K}$ is a class of $\nu$-algebras we set
\begin{align*}
	\op{ThR}(\class{K})&:=\{\Gamma/\Delta\in \mathit{Rul}_\nu:\class{K}\vDash \Gamma/\Delta\}.\\
\end{align*}

A \emph{variety} (resp.\ \emph{universal class}) of $\nu$-algebras is a class of $\nu$-algebras closed under homomorphic images, subalgebras and direct products (resp.\ under isomorphic copies, subalgebras and ultraproducts). When $\class{K}$ is a class of $\nu$-algebras we write $\lat{Var}(\class{K})$ and $\lat{Uni}(\class{K})$ respectively for the class of subvarieties and of universal subclasses of $\class{K}$. It is well known that both $\lat{Var}(\class{K})$ and $\lat{Uni}(\class{K})$ admit the structure of a complete lattice. 
The meet operations of $\lat{Var}(\class{K})$ and $\lat{Uni}(\class{K})$, denoted  $\otimes_{\lat{Var}(\class{K})}$ and $\otimes_{\lat{Uni}(\class{K})}$ respectively, coincide. However, the joins $\oplus_{\lat{Var}(\class{K})}$ and $\oplus_{\lat{Uni}(\class{K})}$ generally come apart. We can characterize $\oplus_{\lat{Var}(\class{K})}$ in terms of $\oplus_{\lat{Uni}(\class{K})}$. For $\class{U}\in \lat{Uni}(\class{K})$, let $\op{Var}(\class{U})$ be the least variety in which $\class{U}$ is contained. 
\begin{proposition}
    Let $\class{K}$ be a class of $\nu$-algebras. Then the identity
\[\oplus_{\lat{Var}(\class{K})}\{\class{V}_i:i\in I\}=\op{Var}(\oplus_{\lat{Uni}(\class{K})}\{\class{V}_i:i\in I\})\]
   holds for any $\{\class{V}_i:i\in I\}\subseteq \lat{Var}(\class{K})$.\label{univarieties}
\end{proposition}
\begin{proof}
    Analogous to that of \Cref{rulesystemlogic}.
\end{proof}

Throughout the paper, we study the structure of lattices of rule systems via semantic methods. This is made possible by the following fundamental result connecting the syntactic types of rule system to closure conditions on the classes of algebras that validate them. \Cref{birkhoff} is widely known as \emph{Birkhoff's theorem}, after \cite{Birkhoff1935OtSoAA}.
\begin{theorem}[{\cite[Theorems II.11.9 and V.2.20]{BurrisSankappanavar1981ACiUA}}]
	For every class $\class{K}$ of $\nu$-algebras, the following conditions hold:\label{syntacticvarietiesuniclasses}
	\begin{enumerate}
		\item $\class{K}$ is a variety iff $\class{K}=\op{Alg}(\logic{S})$ for some set of $\nu$-formulae $\logic{S}$. \label{birkhoff}
		\item $\class{K}$ is a universal class iff $\class{K}=\op{Alg}(\logic{S})$ for some set of $\nu$-rules $\logic{S}$.
	\end{enumerate}
\end{theorem} 
In this sense, $\nu$-logics correspond to varieties of $\nu$-algebras, whereas $\nu$-rule systems correspond to universal classes of $\nu$-algebras.

We now briefly describe the classes of algebras we shall use to interpret the rule systems under discussion in more detail, and review some of their basic properties. For further details on these structures, we point the reader to \cite{Esakia2019HADT,ChagrovZakharyaschev1997ML,Rauszer1974SBAaTAtILwDO,PedrosoDeLimaMartins2021BGAaCT,Kowalski1998VoTA,Venema2007AaC,Esakia2006TMHCaCMEotIL}.

A \emph{Heyting algebra} is a tuple $\alg{H}=(H, \land, \lor, \to, 0, 1)$ such that $(H, \land, \lor, 0, 1)$ is a bounded distributive lattice and for every $a, b, c\in A$ we have 
\[c\leq a\to b\iff a\land c\leq b.\]
A \emph{bi-Heyting algebra} is a tuple $\alg{H}=(H, \land, \lor, \to,\from, 0, 1)$ such that the $\from$-free reduct of $\alg{H}$ is a Heyting algebra, and such that for all $a, b, c\in H$ we have
\[a\from b\leq c\iff a\leq b\lor c.\]
Equivalently, a bi-Heyting algebra can be defined as a Heyting algebra $\alg{H}$ whose order dual is also a Heyting algebra, whose implication is defined by the identity
\[a\from b:=\bigwedge\{c\in H:a\leq b\lor c\}.\]

A \emph{modal} algebra is a tuple $\alg{M}=(M, \land, \lor, \neg, \square, 0, 1)$ such that $(M, \land, \lor, \neg, 0, 1)$ is a Boolean algebra and the following equations hold:
\begin{align}
	 \square 1&=1,\label{norm1} \\
	 \square(a\land b)&=\square a\land \square b.\label{norm2}
\end{align}
In any modal algebra $\alg{M}$ we can define the compound modality 
\begin{equation}
	\square^+a:=\square a \land a.\label{compound}
\end{equation}

A \emph{tense algebra} is a structure $\alg{M}=(M, \land, \lor, \neg, \square, \pastlog, 0, 1)$, such that both the $\square$-free and the $\pastlog$-free reducts of $\alg{M}$ are both modal algebras (the former with respect to the dual of $\pastlog$), and $\square, \pastlog$ form a residual pair. That is, for all $a, b\in M$ we have the following identity:
\begin{equation}
	\pastlog a \leq b\iff a\leq \square y.
\end{equation}

Finally, a \emph{frontal Heyting algebra} is a structure $\alg{H}=(H, \land, \lor, \to,\sisquare,  0, 1)$ whose $\sisquare$-free reduct is a Heyting algebra and such that $\boxtimes$ satisfies the identities (\ref{norm1}) and (\ref{norm2}), as well as the following inequalities:
\begin{align}
	a&\leq \boxtimes a,\\
	\boxtimes a&\leq b\lor (b\to a).
\end{align}

We write $\op{HA}, \op{biHA}, \op{MA}, \op{Ten}, \op{FHA}$ for the classes of Heyting algebras, bi-Heyting algebras, modal algebras, tense algebras and frontal Heyting algebras respectively. It is well known that all these classes are equationally definable, hence varieties by \Cref{syntacticvarietiesuniclasses}. What is more, their universal subclasses are algebraic counterparts of the rule systems introduced in the previous subsection, in the sense spelled out by the following theorem. 
\begin{theorem}
	The following maps are pairs of mutually inverse dual isomorphisms:\label{algebraization}
	\begin{itemize}
		\item $\op{Alg}:\lat{Ext}(\logic{IPC})\to \lat{Uni}(\op{HA})$ and $\op{ThR}:\lat{Uni}(\op{HA})\to \lat{Ext}(\logic{IPC})$.\label{algebraisationipc}
		\item $\op{Alg}:\lat{Ext}(\logic{biIPC})\to \lat{Uni}(\op{biHA})$ and $\op{ThR}:\lat{Uni}(\op{biHA})\to \lat{Ext}(\logic{biIPC})$.\label{algebraisationbiipc}
		\item $\op{Alg}:\lat{NExt}(\logic{K})\to \lat{Uni}(\op{MA})$ and $\op{ThR}:\lat{Uni}(\op{MA})\to \lat{NExt}(\logic{K})$.\label{algebraisationk}
		\item $\op{Alg}:\lat{NExt}(\logic{K.t})\to \lat{Uni}(\op{Ten})$ and $\op{ThR}:\lat{Uni}(\op{Ten})\to \lat{NExt}(\logic{K.t})$.\label{algebraisationkt}
		\item $\op{Alg}:\lat{NExt}(\logic{IPCK})\to \lat{Uni}(\op{FHA})$ and $\op{ThR}:\lat{Uni}(\op{FHA})\to \lat{NExt}(\logic{IPCK})$.\label{algebraisationfrtuni}
	\end{itemize}
     Furthermore, the items above remain true when we substitute $\lat{NExtL}$ $($resp.\ $\lat{ExtL})$ for $\lat{NExt}$ $($resp.\ $\lat{Ext})$ and $\lat{Var}$ for $\lat{Uni}$.
\end{theorem}
\begin{corollary}
	Every si (resp.\ bsi, modal, tense, msi) rule system $\logic{L}$ is complete with respect of some universal class of Heyting (resp.\ bi-Heyting, modal, tense, frontal Heyting) algebras. Moreover, if $\logic{L}$ is a logic, then (by \Cref{syntacticvarietiesuniclasses}) it is complete with respect to a variety of algebras of the appropriate kind. \label{completeness}
\end{corollary}


Lastly, we introduce some uniform notation to refer to the \emph{non truth-functional operations} of a $\nu$-algebra. For $\alg{A}$  a $\nu$-algebra, let 
\[\ops{\alg{A}}:=\begin{cases}
	\{\to \} &\text{if }\alg{A}\in \op{HA}\\
	\{\to, \from \} &\text{if }\alg{A}\in \op{biHA}\\
	\{\square \} &\text{if }\alg{A}\in \op{MA}\\
	\{\square, \pastlog \} &\text{if }\alg{A}\in \op{Ten}\\
	\{\to, \sisquare \} &\text{if }\alg{A}\in \op{FHA}
\end{cases}\]

	\subsection{Geometric Semantics and Duality} 
	
	All the rule systems mentioned so far also admit a more suggestive geometric-topological semantics, which we shall rely on in the proofs of several results. We sketch this semantics here and relate the basic topological structures it involves to their algebraic counterparts. 

	A \emph{Stone space} is a compact Hausdorff space with a basis of clopens. The topological structures we shall work with are all expansions of Stone spaces with one or more binary relations satisfying various conditions. For each of the signatures $\nu$ presented earlier, there is a corresponding class of such spaces, which for the moment we call $\nu$-\emph{spaces}. When $\spa{X}:=(X, \preceq_1, \ldots, \preceq_n, \class{O})$ is a $\nu$-space  we let $\clop{X}$ denote the set of clopen subsets of $\spa{X}$, and let $\op{ClopUp}_{\preceq_i}(\spa{X})$ denote the set of clopen upsets of $\spa{X}$ with respect to the relation $\preceq_i$, i.e., those elements of $\clop{X}$ which are upward-closed with respect to the relation $\preceq_i$. Moreover, for $U\subseteq X$ we write
	\begin{align*}
		\upset[\preceq_i]{U}&:=\{x\in X: y\preceq_i x\text{ for some }y\in U\},\\
		\downset[\preceq_i]{U}&:=\{x\in X: x\preceq_i y\text{ for some }y\in U\}.
	\end{align*}
	In case $U=\{x\}$ we write $\upset[\preceq_i]{x}$ and $\downset[\preceq_i]{x}$ instead of $\upset[\preceq_i]{\{x\}}$ and $\downset[\preceq_i]{\{x\}}$. 
	When the space in question is only equipped with one relation or when the relation in question is clear from context, we may omit the subscripts from any of these operations. 
	
	We now describe these spaces in more detail. An \emph{Esakia space} is a triple $\spa{X}=(X, \leq, \class{O})$ such that $(X, \class{O})$ is a Stone space and $\leq$ is a partial order satisfying the following conditions:
	\begin{enumerate}
		\item $\upset{x}$ is closed for every $x\in X$;\label{esa1}
		\item $\downset{U}\in \clop{X}$ for every $U\in \clop{X}$. \label{esa2}
	\end{enumerate}
	If, in addition, the structure $\spa{X}^{-1}=(X, \geq, \class{O})$ is also an Esakia space, where $\geq$ is the converse of $\leq$, then we call $\spa{X}$ a \emph{bi-Esakia space}. 

	A \emph{modal space} is a triple $\spa{X}=(X, R, \class{O})$ such that $(X, \class{O})$ is a Stone space and $R$ is a binary relation---not necessarily reflexive and transitive---satisfying conditions (\ref{esa1}) and (\ref{esa2}) above. When the structure $\spa{X}^{-1}=(X, \breve{R}, \class{O})$ is also a modal space, where $\breve{R}$ is the converse of $R$, we call $\spa{X}$ a \emph{tense space}. 
	
	Finally, a \emph{modalized Esakia space} is a quadruple $\spa{X}=(X, \leq, \sqsubseteq, \class{O})$ such that $(X, \leq,  \class{O})$ is an Esakia space and the following conditions hold:
	\begin{enumerate}
		\item $\{x\in X: \upset[\sqsubseteq]{x}\subseteq U\}\in \clopup[\leq]{X}$ whenever $U\in \clopup[\leq]{X}$; 
		\item The reflexive closure of $\sqsubseteq$ coincides with $\leq$. 
	\end{enumerate}

	Let $\spa{X}=(X, \preceq_1, \ldots, \preceq_n, \class{O})$ and  $\spa{X}'=(X', \preceq'_1, \ldots, \preceq'_n, \class{O}')$ be $\nu$-spaces. A mapping $f:\spa{X}\to \spa{X}'$ is called a \emph{$\nu$ bounded morphism} when it is continuous and satisfies the conditions below for all $x, y\in X$ and each $i\leq n$:
	\begin{enumerate}
		\item $x\preceq_i y$ only if $f(x)\preceq'_i f(y)$;
		\item $f(x)\preceq'_i f(y)$  only if there is $z\in f^{-1}(y)$ such that $x\preceq_i z$.
	\end{enumerate}
	In the special case where $\nu\in \{\mathit{bsi, ten}\}$, we must, in addition, require that the above conditions hold for the converses of $\preceq, \preceq'$. 

	We now describe how to interpret $\nu$-rule systems over $\nu$-spaces. Let $\spa{X}$ be a $\nu$-space. If  $\nu\in \{\mathit{si, bsi, msi}\}$, a \emph{$\nu$-valuation} on $\spa{X}$ is a mapping $V:\mathit{Frm}_\nu\to \clopup[\leq]{X}$ that commutes with the connectives in $\nu$ in the usual way. On the other hand, if $\nu\in \{\mathit{md}, \mathit{ten}\}$, a \emph{$\nu$-valuation} on $\spa{X}$ is defined in a similar way, except that we require $V$ to range over $\clop{X}$ instead of $\clopup{X}$. We list below how valuations commute with the most important connectives. 
	\begin{align}
		V(\varphi\to \psi)&=-\downset[\leq]{(V(\varphi)\smallsetminus V(\psi))},\\
		V(\varphi\from \psi)&=\upset[\leq]{(V(\varphi)\smallsetminus V(\psi))},\\
		V(\sisquare \varphi)&=\{x\in X: \upset[\sqsubseteq]{x}\subseteq V(\varphi)\}, \\
		V(\square \varphi)&=\{x\in X: \upset[R]{x}\subseteq V(\varphi)\}.\\
		V(\pastlog \varphi)&=\upset[R]{V(\varphi)}.
	\end{align}
	Here and throughout the paper, we use $-$ and $\smallsetminus$ to denote, respectively, the set-theoretic relative complement and difference operations. 

	Let $\spa{X}$ be a $\nu$-space and $V$ a valuation on it. A formula $\varphi$ is \emph{satisfied} on a model $(\spa{X}, V)$ at a point $x$ if $x\in V(\varphi)$. In this case we write $\spa{X}, V, x\vDash \varphi$, otherwise we write $\spa{X}, V, x\nvDash \varphi$ and say that the model $(\spa{X}, V)$ \emph{refutes} $\varphi$ at a point $x$. A rule $\Gamma/\Delta$ is \emph{valid} on a model $(\spa{X}, V)$ when the following holds: if $\spa{X}, V, x\vDash \gamma$ holds for each $x\in X$ and every $\gamma\in \Gamma$, then there is some $\delta\in \Delta$ such that   $\spa{X}, V, x\vDash \delta$ holds for each $x\in X$. In this case we write $\spa{X}, V\vDash \Gamma/\Delta$, otherwise we write $\spa{X}, V\nvDash \Gamma/\Delta$ and say that the model $(\spa{X}, V)$ \emph{refutes} $\varphi$. A rule $\Gamma/\Delta$ is \emph{valid} on a $\nu$-space $\spa{X}$ if it is valid on the model $(\spa{X}, V)$ for every valuation $V$ on $\spa{X}$, otherwise $\spa{X}$ \emph{refutes} $\Gamma/\Delta$. We write $\spa{X}\vDash \Gamma/\Delta$ to mean that $\Gamma/\Delta$ is valid on $\spa{X}$, and $\spa{X}\nvDash\Gamma/\Delta$ to mean that $\spa{X}$ refutes $\Gamma/\Delta$.  The notion of validity generalizes to classes of $\nu$-spaces, as well as to classes of rules, in the obvious way.


	For each of the signatures $\nu$ we shall work with, there is a duality result connecting $\nu$-algebras to $\nu$-spaces. All these dualities are generalizations of Stone duality, which relates the category of Boolean algebras with homomorphisms to that of Stone spaces with continuous functions \cite{Johnstone1982SS}. We list these dualities in the following theorem. 
	\begin{theorem}
		The category of modal (resp.\ Heyting, tense, bi-Heyting, frontal Heyting) algebras with homomorphisms is dually equivalent to the category of modal (resp.\ Esakia, tense, bi-Esakia, modalised Esakia) spaces with bounded morphisms. \label{topodual}
	\end{theorem}
	In each of these cases, we write $\dualspa{A}$ for the space dual to an algebra $\alg{A}$ and 
	$\dualalg{X}$ for the algebra dual to a space $\spa{X}$. The space $\dualspa{\alg{A}}$ is always an expansion of the space of prime filters of $\alg{A}$. We write $\beta$ for the map, called the \emph{Stone map}, which takes element $a$ and returns the set $\beta(a)$ of prime filters in that algebras that contain $a$. In the other direction, the algebra $\dualalg{\spa{X}}$ is constructed by taking clopen sets (if $\nu\in \{\mathit{md, ten}\}$) or clopen upsets (if $\nu\in \{\mathit{si, bsi, msi}\}$). We refer the reader to \cite{Esakia1974TKM,SambinVaccaro1988TaDiML,Esakia1975TPoDitILaBL,CastiglioniEtAl2010OFHA} for detailed descriptions and proofs of these dualities. 

	\subsection{Kripke semantics}
	Besides spaces, in \Cref{sec:scr,sec:mc} we shall also work with Kripke frames. We will only use Kripke frames to interpret si, bsi, modal and tense rule systems. Thus we define a \emph{Kripke frame} to be a set $\spa{X}:=(X, \preceq)$, where $X$ is a non-empty set and $\preceq$ a binary relation on $X$. An \emph{intuitionistic Kripke frame} is a Kripke frame $\spa{X}:=(X, \leq)$, where $\leq$ is a partial order. The notions of \emph{$\nu$ bounded morphism} for $\nu\in \{\mathit{md, ten, si, bsi}\}$ are defined the same way as for spaces, but omitting the requirement of continuity. 

	For $\nu\in \{\mathit{md}, \mathit{ten}\}$, a \emph{$\nu$-valuation} on a Kripke frame $\spa{X}$ is a mapping $V:\mathit{Frm}_\nu\to \wp(X)$ that commutes with the connectives in $\nu$ in the usual way. For $\nu\in \{\mathit{si, bsi}\}$, \emph{$\nu$-valuation} on an intuitionistic Kripke frame $\spa{X}$ is a mapping $V:\mathit{Frm}_\nu\to \wp(X)$ that commutes with the connectives in $\nu$ in the usual way, such that $\upset[]{V}(\varphi)=V(\varphi)$ for every $\varphi\in\mathit{Frm}_\nu$. We extend our notions of satisfaction and validity from spaces to Kripke frames in the obvious way. 

	We recall briefly the following duality results concerning Kripke frames, which were first proved, respectively, in \cite{Thomason1975CoFfML} and \cite{Jongh1966OtCoPOSwSPBA} (see also \cite{Litak2005AAAtIiML}).
	\begin{theorem}
		The following categories are dually equivalent. \label{kripkedual}
		\begin{enumerate}
			\item Kripke frames with modal (resp.\ tense) bounded morphisms and complete, atomic, completely additive and completely distributive modal (resp.\ tense) algebras with complete homomorphisms.
			\item Intuitionistic Kripke frames with si (resp.\ bsi) bounded morphisms and complete, completely distributive and completely join prime generated Heyting (resp.\ bi-Heyting) algebras with complete homomorphisms. \label{kripkedual-h}
		\end{enumerate}
	\end{theorem}
	We write $\dualkr{A}$ for the Kripke frame dual to an algebra $\alg{A}$ among those mentioned in the Theorem above, and 
	$\dualcalg{X}$ for the algebra dual to a Kripke frame $\spa{X}$. The signature of $\dualcalg{X}$ will be clear from context. The Kripke frame $\dualkr{A}$ is constructed by expanding the set of \emph{completely join prime} filters of $\alg{A}$ with a binary relation, which is defined the same way as in the dualities from \Cref{topodual}. The algebra $\dualcalg{X}$ is constructed the same way as $\dualalg{Y}$ when $\spa{Y}$ is a space, but taking subsets (resp.\ upsets) instead of clopen subsets (resp.\ clopen upsets).

	\begin{convention}
		Before moving on, we introduce a notational convention we shall use throughout the paper to discuss related rule systems and structures while avoiding cumbersome repetitions. The convention consists in the use of parentheticals in expressions naming rule systems, mathematical structures and classes thereof. For example, we will use the expression `$\logic{S4(.t)}$' to refer simultaneously to the rule systems $\logic{S4}$ and $\logic{S4.t}$. Similarly, we will use the expression `(bi-)Heyting algebras' to refer simultaneously to Heyting and bi-Heyting algebras. 

		We use these parentheticals in the same way parentheticals of the form ``(resp.\ \ldots)'' are normally used.  To illustrate, \Cref{kripkedual-h} in \Cref{kripkedual} can be rewritten, using the convention just introduced, in the following way: 
		\begin{quote}
			The following categories are dually equivalent: intuitionistic Kripke frames with (b)si bounded morphisms and complete, completely join prime generated (bi-)Heyting algebras with complete homomorphisms.
		\end{quote}
	\end{convention}
	 
	

	\subsection{Transitive structures}

	We close our preliminaries by reviewing some classes of transitive structures we shall encounter throughout the paper. Let us first introduce some more notational conventions. We refer to an algebra in $\op{Alg}(\logic{S})$ as an $\logic{S}$-algebra. Similarly, we let an $\logic{S}$-space be a space in $\op{Spa}(\logic{S})$, and an $\logic{S}$-frame be a Kripke frame that validates every rule in $\logic{S}$---with the additional requirement that an  $\logic{S}$-frame be \emph{intuitionistic} when $\logic{S}$ is a (b)si logic.

	We recall that the $\logic{K4(.t)}$-spaces can be characterized as those modal (resp.\ tense) spaces with a transitive relation, and that the $\logic{S4(.t)}$-spaces coincide with those $\logic{K4(.t)}$-spaces where the relation is, in addition, reflexive. We recall some well known properties of these spaces. Given a preordered set $(X, R)$, we define:
	\begin{align*}
		\mathit{qmax}_{R}(U)&:=\{x\in U: \text{for all $y\in U$, if $Rxy$, then $Ryx$}\}\\
		\mathit{max}_{R}(U)&:=\{x\in U: \upset[R]{x}\subseteq\{x\}\}\\
		\mathit{qmin}_{R}(U)&:=\{x\in U: \text{for all $y\in U$, if $Ryx$, then $Rxy$}\}\\ 
		\mathit{min}_{R}(U)&:=\{x\in U:\downset[R]{x}\subseteq\{x\}\}
	\end{align*}
	We omit subscripts when they can be inferred from context. 
	\begin{proposition}
		Let $\spa{X}$ be a $\logic{K4}$-space. Then the following conditions hold for every $x\in X$ and each $U\in \clop{X}$. \label{propk4}
		\begin{enumerate}
			\item $\mathit{qmax}(U)$ is closed. \label{maxclosed}
			\item If $x\in U$, then either $x\in \mathit{max}(U)$ or there is $y\in \mathit{qmax}(U)$ such that $Rxy$. \label{seemax}
			\item When $\spa{X}$ is a $\logic{S4}$-space, \Cref{seemax} can be strengthened to the following: if $x\in U$, then there is $y\in \mathit{qmax}(U)$ such that $Rxy$. \label{seemaxs4}
			\item When $\spa{X}$ is a $\logic{S4.t}$-space, \Cref{maxclosed,seemaxs4} remain true if we substitute  $\mathit{qmax}(U)$ for  $\mathit{qmin}(U)$ and $Rxy$ for $Ryx$. 
		\end{enumerate}
	\end{proposition}

	Among $\logic{S4(.t)}$-spaces, we shall pay particular attention to $\logic{Grz(.t)}$-spaces. We recall some of their basic properties. Given a preordered set $(X, R)$ and $U\subseteq X$, we call an element $x\in U$ \emph{passive in $U$} when there is no $y\in X\smallsetminus U$ such that $Rxy$ and $Ryz$ for some $z\in U$. In other words, $x$ is passive in $U$ when one cannot ``leave'' and ``re-enter'' $U$ starting from $x$. A \emph{cluster} in $(X, R)$ is a set $C\subseteq X$ which is maximal with the property that $Rxy$ and $Ryx$ whenever $x, y\in C$. A set $U\subseteq{X}$ is said to \emph{cut} a cluster $C\subseteq X$ when neither $C\subseteq U$ nor $C\cap U=\varnothing$ hold. 
	\begin{theorem}[{\cite[Thm. 3.5.5]{Esakia2019HADT}}]
		Let $\spa{X}$ be an $\logic{S4(.t)}$-space. Then $\spa{X}$ is a  $\logic{Grz(.t)}$-space if and only if for every $U\in \clop{X}$ and any $x\in U$, there is a $y\in U$ such that $Rxy$ and $y$ is passive in $U$ $($and there is some $z\in U$ such that $Rzx$ and $z$ is passive in $U$ with respect to the converse of $R)$. \label{grz-charact}
	\end{theorem}
	\begin{corollary}[{\cite[Thm.\ 3.5.6]{Esakia2019HADT}}]
		Let $\spa{X}$ be a $\logic{Grz}$-spaces and $U\in \clop{X}$. The following conditions hold: \label{grz-max}
		\begin{enumerate}
			\item $\mathit{qmax}(U)=\mathit{max}(U)$.
			\item $\mathit{max}(U)$ does not cut any cluster. 
		\end{enumerate}
		Moreover, if $\spa{X}$ is also a $\logic{Grz.t}$-space, the conditions above continue to hold when we substitute $\mathit{qmin}(U)$ for $\mathit{qmax}(U)$ and $\mathit{min}(U)$ for $\mathit{max}(U)$.  
	\end{corollary}
	\begin{corollary}[{\cite[Thm.\ 3.5.8]{Esakia2019HADT}}]
		Let $\spa{X}$ be a $\logic{S4(.t)}$-space. If $\spa{X}$ is partially ordered, then $\spa{X}$ is a $\logic{Grz(.t)}$-space. \label{grz-skeletal}
	\end{corollary}

	We mention another simple fact concerning clusters which we shall appeal to several times. 
	\begin{proposition}
		Let $\spa{X}, \spa{Y}$ be $\logic{S4(.t)}$-spaces or Kripke frames and let $f:\spa{X}\to \spa{Y}$ be an order-preserving map. Then $f^{-1}(U)$ does not cut clusters for any $U\subseteq{Y}$. \label{cutclusters}
	\end{proposition}

	Another class of $\logic{K4}$-spaces we shall pay close attention to is the class of $\logic{GL}$-spaces. These spaces display various similarities with $\logic{Grz}$-spaces, as the reader can appreciate by comparing the following results with \Cref{propk4,grz-charact,grz-max}.

    \begin{theorem}
        Let $\spa{X}$ be a $\logic{GL}$-space. Then $\spa{X}$ is a $\logic{GL}$-space if and only if for every $U\in \clop{X}$ and any $x\in X$, if $\upset{x}\cap U\neq \varnothing$, then there is some $y\in U$ such that $Rxy$ and $\upset{y}\cap U=\varnothing$.
    \end{theorem}
	\begin{corollary}
		Let $\spa{X}$ be a $\logic{GL}$-space and $U\in \clop{X}$. The following conditions hold: \label{propgl}
		\begin{enumerate}
			\item $\mathit{max}(U)=\{x\in U: \upset{x}\cap U=\varnothing\}$;
			\item $\mathit{max}(U)\in \clop{X}$;
			\item If $x\in U$, then either $x\in \mathit{max}(U)$ or there is $y\in \mathit{max}(U)$ such that $Rxy$.
		\end{enumerate}
	\end{corollary}
    \begin{corollary}
        Let $\spa{X}$ be a $\logic{K4}$-space. If $\spa{X}$ has an irreflexive relation, then $\spa{X}$ is a $\logic{GL}$-space. \label{gl-skeletal}
    \end{corollary}

    \section{Mappings and Translations}
    \label{sec:mappings}

	The main results discussed in this paper all involve translations between rules in different signatures, and semantic transformations corresponding to them. The purpose of this section is to introduce these translations and transformations. 

 \begin{convention}
     To treat these mappings uniformly, we introduce some notational conventions to refer to the three pairs of signatures which we want to connect via translations. We let the numerals $\simod$, $\bsiten$ and $\msimod$ denote, respectively, the pairs of signatures $(\mathit{si}, \mathit{mod})$, $(\mathit{bsi}, \mathit{ten})$ and $(\mathit{msi}, \mathit{mod})$. When $\sig$ is any of these pairs of signatures, the signature occurring in the first coordinate of $\sig$ is called the \emph{intuitionistic} signature, whereas the signature occurring in the second coordinate of $\sig$ is called the \emph{classical} signature. \label{signatureconvention}
 \end{convention}

	For each pair of signatures $\sig\in \{\simod, \bsiten, \msimod\}$ we will define a translation function $T_\sig$, as well as algebraic, topological and syntactic versions of three mappings, $\greatest[\sig], \least[\sig]$ and $\fragment[\sig]$. In the case of signature pairs $\simod$ and $\bsiten$ we will also define versions of the maps $\greatest[\sig]$ and $\fragment[\sig]$ on Kripke frames. We will adopt the further convention of suppressing subscripts for signature pairs when they can be inferred from context. 

	We defined distinguished rule systems and universal classes of algebras for each pair of signatures $\sig\in \{\simod, \bsiten, \msimod\}$, as follows:
	\begin{gather}
		\logic{I}_\sig:=\begin{cases}
			\logic{IPC}& \text{if $\sig=\simod$}\\
			\logic{biIPC}&\text{if $\sig=\bsiten$}\\
			\logic{KM}&\text{if $\sig=\msimod$}
		\end{cases}\quad
		\logic{C}_\sig:=\begin{cases}
			\logic{S4}& \text{if $\sig=\simod$}\\
			\logic{S4.t}&\text{if $\sig=\bsiten$}\\
			\logic{K4}&\text{if $\sig=\msimod$}
		\end{cases}\quad
		\logic{C}^+_\sig:=\begin{cases}
			\logic{Grz}& \text{if $\sig=\simod$}\\
			\logic{Grz.t}&\text{if $\sig=\bsiten$}\\
			\logic{GL}&\text{if $\sig=\msimod$}
		\end{cases}
		\label{baselogics}\\
	\class{I}_\sig:=\op{Alg}(\logic{I}_\sig)\quad\class{C}_\sig:=\op{Alg}(\logic{C}_\sig)\quad \class{C}^+_\sig:=\op{Alg}(\logic{C}^+_\sig). \label{basealgebras}
	\end{gather}
	We shall use this notation to state definitions and results concerning the mappings mentioned above in a uniform fashion. 

    \subsection{Mappings on Algebras}

    We begin by reviewing some well-known semantic transformations between algebras.
    Recall that the \emph{free Boolean extension} of a Heyting algebra $\alg{H}$ is the unique Boolean algebra $B(\alg{H})$ in which $\alg{H}$ embeds as a distributive lattice, such that the image of $\alg{H}$ under this embedding generates $B(\alg{H})$ as a Boolean algebra [\citealp[Def.\ 2.5.6, Constr.\ 2.5.7]{Esakia2019HADT}; \citealp[Sec.\ V]{BalbesDwinger1975DL}]. For simplicity, we will generally identify $\alg{H}$ with its image in $B(\alg{H})$. Note this convention is used in the definitions to follow.

    If $\alg{H}$ is a Heyting algebra, the algebra $\greatest[\simod]\alg{H}$ is constructing by expanding  $B(\alg{H})$   with the operation 
	\[\square a:=\bigvee \{b\in H: b\leq a\}.\]
	If $\alg{M}$ is a bi-Heyting algebra, we define $\greatest[\bsiten]\alg{H}$ the same way but also add the operation 
	\[\pastlog a:=\bigwedge \{b\in H: a\leq b\}.\]
    Finally, if $\alg{M}$ is a frontal Heyting algebra, we define $\greatest[\msimod]\alg{H}$ by expanding $B(\alg{H})$ with the operation
    \[
    \square a:=\boxtimes Ia,
    \]
    where 
    \[
    Ia:=\bigvee\{b\in H:b\leq a\}.
    \]
	It is  known that $\greatest(\alg{H})$ is an $\logic{Grz(.t)}$-algebra  whenever $\alg{H}$ is a (bi-)Heyting algebra [\citealp[Cor. 3.5.9]{Esakia2019HADT}; \citealp[Lem. 16]{Wolter1998OLwC}], and moreover that $\greatest[\msimod](\alg{H})$ is a $\logic{GL}$-algebra whenever $\alg{H}$ is a frontal Heyting algebra \cite[Cor. 20]{Esakia2006TMHCaCMEotIL}. 

	Conversely, if $\alg{M}$ is an $\logic{S4}$-algebra,  the algebra $\fragment[\simod] \alg{M}$ is constructed as follows. As the carrier we take the bounded lattice $O(\alg{M})$ of open elements of $\alg{M}$, that is, of those elements $a\in M$ with $\square a=a$, or, equivalently, $\pastlog a=a$. We expand this lattice with the operation
	\[a\to b:=\square (\neg a\lor b).\]
	When $\alg{M}$ is an $\logic{S4.t}$-algebra, we define $\fragment[\bsiten] \alg{M}$ the same way but also add the operation 
	\[a\from b:=\pastlog ( a\land \neg b).\]
    Likewise, if $\alg{M}$ is a $\logic{K4}$-algebra, the algebra $\fragment[\msimod]\alg{M}$ is constructed as follows. As the carrier we take the bounded lattice $O^+(\alg{M})$ of \emph{quasi}-open elements of $\alg{M}$, i.e., those elements of $\alg{M}$ with $\square^+a=a$, where $\square^+a:=\square a\land a$. We then expand this lattice with the following operations:
    \begin{align*}
        a\to b&:=\square^+ (\neg a\lor b)\\
        \boxtimes a&:= \square a,
    \end{align*}
    It is well known that $\fragment\alg{M}$ is a (bi-)Heyting algebra for every $\logic{S4(.t)}$-algebra $\alg{M}$, and that $\fragment[\msimod]\alg{M}$ is a frontal Heyting algebra for every $\logic{K4}$-algebra $\alg{M}$. 

    All these mappings can be lifted to universal classes. Given $\sig\in \{\simod, \bsiten, \msimod\}$, let $\class{U, V}$ be universal classes of algebras on which, respectively, $\greatest[\sig]$ and $\fragment[\sig]$ are defined. We then put 
	\[\greatest[\sig] \class{U}:=\op{Uni}\{\greatest[\sig] \alg{H}:\alg{H}\in \class{U}\}\qquad \fragment[\sig] \class{V}:= \{\fragment[\sig]\alg{A}:\alg{A}\in \class{V}\}.\]
	We also introduce mappings $\least[\sig]:\op{Uni}(\class{I}_\sig)\to \op{Uni}(\class{C}_\sig)$  by setting 
	\[\least[\sig] \class{W}:=\{\alg{M}\in\class{C} :\fragment[\sig]\alg{M}\in \class{W}\}. \]

    \subsection{Mappings on Spaces}

    We now describe the maps defined in the previous subsection dually. If $\spa{X}$ is an Esakia space, we set 
	\[\greatest[\simod] \spa{X}:=(X, R, \mathcal{O})\qquad\qquad  R:= {\leq}.\] 
	When $\spa{X}$ is a bi-Esakia space we let $\greatest[\bsiten]\spa{X}:=\greatest[\simod]\spa{X}$. 
	Thus $\greatest[\simod]$ and $\greatest[\bsiten]$ are just  identity maps; we simply notate the relation differently to indicate that we are viewing the structure as a modal or tense space. Moreover, if $\spa{X}$ is a modalized Esakia space, we let $\greatest[\msimod]\spa{X}$ be the $\leq$-free reduct of $\spa{X}$. By \Cref{grz-skeletal}, if $\spa{X}$ is a (bi-)Esakia space, then $\greatest\spa{X}$ is always a $\logic{Grz(.t)}$-space.  Likewise, by \Cref{gl-skeletal}, $\greatest[\msimod]\spa{X}$ is always a $\logic{GL}$-space whenever $\spa{X}$ is a modalized Esakia space. 
    
    Conversely, let  $\spa{Y}:=(Y, R, \mathcal{O})$ be a $\logic{K4}$-space. For $x, y\in Y$ write $x\sim y$ iff $Rxy$ and $Ryx$. Define a map $\varrho:Y\to \wp (Y)$ by setting $\varrho(x)=\{y\in Y:x\sim y\}$. We call this map the \emph{skeleton map}. 
	When $\spa{Y}$ is an $\logic{S4}$-space  we let $\fragment[\simod]\spa{Y}:=(\varrho[Y], \leq, \varrho[\class{O}])$ where $\varrho(x)\leq \varrho(y)$ iff $Rxy$. We let $\fragment[\bsiten]$ be the restriction of $\fragment[\simod]$ to $\logic{S4.t}$ spaces. When $\spa{Y}$ is a $\logic{K4}$-space we let $\fragment[\msimod]\spa{Y}:=(\varrho[Y], \leq, \sqsubseteq, \varrho[\class{O}])$, where $\varrho(x)\leq\varrho(y)$ iff $R^+xy$ and $\varrho(x)\sqsubseteq \varrho(y)$ iff $Rxy$. Here $R^+$ denotes the reflexive closure of $R$. 
    
    In other words, when $\spa{Y}$ is a $\logic{S4}$-space, the space $\fragment[\simod]\spa{Y}$ is obtained by collapsing the clusters of $\spa{Y}$, lifting the relation $R$ clusterwise and endowing the result with the quotient topology under the mapping $\varrho$. When $\spa{Y}$ is a $\logic{K4}$-space, $\fragment[\msimod]\spa{Y}$ is constructed in a similar way, except that the intuitionistic relation of $\fragment[\msimod]\spa{Y}$ is obtained by lifting the reflexive closure of $R$ clusterwise, and the modal relation  is obtained by lifting $R$ itself clusterwise. 

    We note a simple property of the transformations $\fragment[\sig]$, which we shall appeal to later on. 
    \begin{proposition}
        Let $\spa{X}$ be a $\logic{K4}$-space. If $U\subseteq X$ is open (resp.\ closed), then $\varrho[U]$ is open (resp.\ closed) in $\fragment[\msimod]\spa{X}$. Moreover, the same holds when $\spa{X}$ is an $\logic{S4}$-space and $\fragment[\msimod]$ is replaced with $\fragment[\simod]$. \label{rhoproperties}
    \end{proposition}

    Routine arguments show that the transformations just defined are indeed dual to their algebraic counterparts defined in the previous subsection \cite[see, e.g.,][Prop.\ 3.4.15]{Esakia2019HADT}. This is to say that given $\sig\in \{\simod, \bsiten, \msimod\}$, for any algebras $\alg{H, M}$ on which the algebraic maps  $\greatest[\sig], \fragment[\sig]$ are defined we have 
	$(\greatest[\sig] \alg{H})_*\cong\greatest[\sig] \alg{H}_*$ and $(\fragment[\sig] \alg{M})_*\cong\fragment[\sig] \alg{M}_*$. 
	Consequently, for any spaces $\spa{X, Y}$ on which the geometric maps $\greatest[\sig], \fragment[\sig]$ are defined we have 
	$(\greatest[\sig] \spa{X})^*\cong\greatest[\sig] \spa{X}^*$ and  $(\fragment[\sig] \spa{Y})^*\cong\fragment[\sig] \spa{Y}^*$. 

    By appealing to these dualities, the following Proposition easily follows. 
    \begin{proposition}
		Given $\sig\in \{\simod, \bsiten, \msimod\}$, let $\alg{H}\in \class{I}_\sig$ and $\alg{M}\in \class{C}_\sig$. Then $\alg{H}\cong \fragment[\sig]\greatest[\sig]\alg{H}$ and $\greatest[\sig]\fragment[\sig]\alg{M}$ is (isomorphic to) a subalgebra of $\alg{M}$.  \label{cor:representationHAS4}
    \end{proposition}

    \subsection{Mappings on Kripke Frames}

    For $\sig\in \{\simod, \bsiten\}$, we also define versions of the maps $\greatest[\sig], \fragment[\sig]$ on Kripke frames. When $\spa{X}$ is an intuitionistic Kripke frame, we let $\greatest[\sig] \spa{X}=\spa{X}$. Conversely, if $\spa{X}$ is a Kripke frame with a reflexive and transitive relation, we let $\fragment[s]\spa{X}$ be defined as we did above for spaces, but omitting the conditions concerning topology. Since the maps are defined the same way for the two pairs of signatures $\simod$ and $\bsiten$, we will always omit signature subscripts when working with Kripke frames. We do not define counterparts of the maps $\greatest[\msimod], \fragment[\msimod]$ on Kripke frames. 

	The $\fragment$ transformation on Kripke frames corresponds quite closely with its topological version. In particular, for every Kripke frame $\spa{F}$ on which the mapping $\fragment$ is defined we have $\dualcalg{(\fragment\spa{F})}\cong \fragment\dualcalg{F}$, and so for every perfect $\logic{S4(.t)}$-algebra $\alg{M}$ we have $\dualkr{(\fragment\spa{M})}\approx \fragment\dualkr{M}$. On the other hand, the algebraic version of the map $\greatest$ fails to preserve atomicity \citep[p. 103]{WolterZakharyaschev2014OtBET}, so in general the identity $\dualcalg{(\greatest F)}\cong \greatest\dualcalg{F}$ may fail. Observe further that when $\spa{F}$ is an intuitionistic Kripke frame, $\greatest\spa{F}$ is not guaranteed to be a Kripke frame for $\logic{Grz(.t)}$. As is well known, the Kripke frames for $\logic{Grz(.t)}$ are precisely those partially ordered Kripke frames which are conversely well founded (resp.\ well-founded and conversely well founded). However, intuitionistic Kripke frames need not be well founded nor conversely well founded. 


    \subsection{Translations}
	\label{sec:trans}

    All the mappings we have introduced are semantic counterparts to various translations between formulas in different signatures. These translations are all versions of the \emph{Gödel Translation} \citep{Goedel1933EIDIA}. For present purposes, we shall define the Gödel translation as a mapping $\transl[\simod]:\mathit{Frm}_{\mathit{si}}\to \mathit{Frm}_{\mathit{md}}$ defined recursively as follows. 
	\begin{align*}
		\transl[\simod](\bot)&:=\bot   	&\transl[\simod](\varphi\lor \psi)&:=\transl[\simod](\varphi)\lor \transl[\simod](\psi)\\
		\transl[\simod](\top)&:=\top 		&\transl[\simod](\varphi\to \psi)&:=\square (\neg \transl[\simod](\varphi)\lor \transl[\simod](\psi))\\
		\transl[\simod](p)&:=\square p 	&\transl[\simod](\varphi\land \psi)&:=\transl[\simod](\varphi)\land \transl[\simod](\psi).
	\end{align*}
    We extend this translation to define two more mappings, $\transl[\bsiten]:\mathit{Frm}_{\mathit{bsi}}\to \mathit{Frm}_{\mathit{ten}}$ and $\transl[\msimod]:\mathit{Frm}_{\mathit{msi}}\to \mathit{Frm}_{\mathit{md}}$. These mappings are obtained by extending the definition above with one of the following additional conditions \citep[cf.][]{Wolter1998FoMLR,KuznetsovMuravitsky1986OSLAFoPLE}:
    \[\transl[\bsiten](\varphi\from \psi):=\pastlog (\transl[\bsiten](\varphi)\land \neg \transl[\bsiten](\psi))\qquad \transl[\msimod](\boxtimes \varphi):=\square \transl[\msimod](\varphi).\]
    $\transl[\bsiten]$ was introduced in \cite{Wolter1998FoMLR}, whereas $\transl[\msimod]$ is equivalent to the translation introduced in \cite{KuznetsovMuravitsky1986OSLAFoPLE} (see also \cite{WolterZakharyaschev1997IMLAFoCBL,WolterZakharyaschev1997OtRbIaCML}).
	We extend these mappings from formulae to rules by setting
	\[\transl[\sig](\Gamma/\Delta):=\transl[\sig][\Gamma]/\transl[\sig][\Delta].\]
    We will refer to all of these mappings as ``Gödel Translations.''  
    
    The interactions between the Gödel translations and the semantic mappings previously introduced are described in the next Lemma. 
    \begin{lemma}[cf. {\cite[Lemma 3.13]{Jerabek2009CR}}]
	Let $\sig\in \{\simod, \bsiten, \msimod\}$ and let $\alg{M}$ be an algebra on which $\fragment[\sig]$ is defined. Then for every rule in the intuitionistic signature of $\sig$, we have \label{lem:gtskeleton}
	\[\alg{M}\models \transl[\sig](\Gamma/\Delta)\iff \fragment[\sig]\alg{M}\models \Gamma/\Delta.\]
    \end{lemma}
    
    Let us now define mappings between logics in different signature by means of the Gödel translations. For $\sig\in \{\simod, \bsiten, \msimod\}$, let $\logic{L}\in \lat{Ext}(\logic{I}_\sig)$. We define:
	\[\least \logic{L}:=\logic{C}_\sig\oplus \{\transl[\sig](\Gamma/\Delta):\Gamma/\Delta\in \logic{L}\} \qquad  \greatest \logic{L}:=\logic{C}^+_\sig\oplus \least[\sig] \logic{L}.\]
    Conversely, if $\logic{M}\in \lat{NExt}(\logic{C}_\sig)$, we put
	\[\fragment[\sig]\logic{M}:=\logic{I}_\sig\oplus \{\Gamma/\Delta: T(\Gamma/\Delta)\in \logic{M}\}.\]

    Finally, let $\logic{L}\in \lat{Ext}(\logic{I}_\sig)$ and $\logic{M}\in \lat{NExt}(\logic{M}_\sig)$. We say that $\logic{M}$ is a \emph{companion} of $\logic{L}$ when $\fragment[\sig]\logic{M}=\logic{L}$. We call the companions of si and msi rule systems \emph{modal companions}, and the companions of bsi rule systems \emph{tense companions}.



	\section{Stable Canonical Rules}
	\label{sec:scr}
	In this section we introduce stable canonical rules for si, bsi, modal and tense rule systems. Essentially, stable canonical rules are syntactic devices for encoding finite filtrations. Although the results of this sections are only discussed in print for the si and modal case, their generalizations to the bsi and tense case are straightforward. We point the reader to \cite{BezhanishviliEtAl2016SCR,BezhanishviliEtAl2016CSL,BezhanishviliBezhanishvili2017LFRoHAaCF,Bezhanishvili2023JFaATfIL} and \cite[Ch. 5]{Ilin2018FRLoSNCL} for more in-depth discussion. 

    We are not going to define stable canonical rules for msi rule systems. This is because the main result we are interested in with respect to msi rule systems is the Kuznetsov-Muravitsky isomorphism, which can be proved using only modal stable canonical rules. We comment on how stable canonical rules for msi rule systems might be developed in \Cref{nomsirule}.


	Since, in this section, we shall not deal with frontal Heyting algebras and their duals, unless otherwise specified we use the term \emph{algebra} to refer to something which is either a modal, tense, Heyting and bi-Heyting algebras without specifying which. We adopt analogous conventions for the terms \emph{space}, \emph{rule}, \emph{rule system}, and so on.

	We begin by defining stable canonical rules. Recall that when $\alg{A}$ is an algebra, by $\ops{\alg{A}}$ we denote the set of non truth-functional operations of $\alg{A}$.
	\begin{definition}
		Let $\alg{H}$ be a finite (bi-)Heyting  algebra and let $D:=(D^\heartsuit)_{\heartsuit \in \ops{\alg{H}}}$, where $D^\heartsuit\subseteq H\times H$. For every $a\in H$ introduce a fresh propositional variable $p_a$. The (si or bsi) \emph{stable canonical rule}  $\scrsi{H}{D}$ is defined as the rule $\Gamma/\Delta$, where
		\begin{align*}
			\Gamma=&\{p_0\leftrightarrow \bot\}\cup\{p_1\leftrightarrow \top\}\cup\\
			&\{p_{a\land b}\leftrightarrow p_a\land p_b:a, b\in H\}\cup \{p_{a\lor b}\leftrightarrow p_a\lor p_b:a, b\in H\}\cup\\
			&\bigcup_{\heartsuit\in \ops{\alg{H}}} \{p_{a\heartsuit b}\leftrightarrow p_a\heartsuit p_b:(a, b)\in D^\heartsuit\}\\
			\Delta=&\{p_a\leftrightarrow p_b:a, b\in H\text{ with } a\neq b\}.
		\end{align*}
	\end{definition}
	\begin{definition}
		Let $\alg{M}$ be a finite modal (resp.\ tense) algebra and let $ D:=(D^\heartsuit)_{\heartsuit \in \ops{\alg{M}}}$, where $D^\heartsuit\subseteq A$. For every $a\in A$ introduce a fresh propositional variable $p_a$. The modal (resp.\ tense) \emph{stable canonical rule}  $\scrmod{A}{D}$ is defined as the rule $\Gamma/\Delta$, where
		\begin{align*}
			\Gamma=&\{p_0\leftrightarrow \bot\}\cup\{p_1\leftrightarrow \top\}\cup\\
			&\{p_{a\land b}\leftrightarrow p_a\land p_b:a, b\in A\}\cup \{p_{a\lor b}\leftrightarrow p_a\lor p_b:a, b\in A\}\cup\\
			&\{p_{\square a}\to \square p_a:a\in A\} \cup \left(\{\pastlog p_a\to p_{\pastlog a}:a\in A\}\cup \right)
			\\
			&\bigcup_{\heartsuit\in \ops{\alg{H}}} \{p_{\heartsuit a}\leftrightarrow\heartsuit p_a:a\in D^\heartsuit\}\\
			\Delta=&\{p_a:a\in A\smallsetminus 1\}.
		\end{align*}
		The parenthetical $\left(\{\pastlog p_a\to p_{\pastlog a}:a\in A\}\cup \right)$, recall, indicates that the formulae $\pastlog p_a\to p_{\pastlog a}$ are to be added to $\Gamma$ only when $\alg{M}$ is a tense algebra.\footnote{Had we taken $\pastsquare$ instead of $\pastlog$ as primitive, we could have given a less disjunctive definition of a tense stable canonical rule. However, the present definition affords a simpler method for transforming tense stable canonical rules based on $\logic{S4}$-algebras into corresponding bsi stable canonical rules.}
	\end{definition}
	We will use the notiation $\scr{A}{D}$ to refer to a stable canonical rule without specifying whether it is modal, tense, si or bsi. 

    \begin{remark}
        To keep the paper relatively short, we decided not to include stable canonical rules for msi rule systems. We can indicate here two ways these might be developed. One straightforward approach is to simply combine modal and si stable canonical rules, requiring partial preservation of $\boxtimes$. This approach is developed in \cite{Liao2023SCRfIML}. Another approach, pursued in 
        \cite{Cleani2021TEVSCR}, is to introduce rules which code up a more general notion of filtration. In the msi setting, this notion of filtration is motivated by the fact that any finite distributive lattice admits a unique expansion to a 
        $\logic{KM}$-algebra. In the modal setting, it is obtained by lifting the requirement that $\square$ be preserved in one direction from the definition of standard filtration. The main reason to work with this more general notion of filtration is that $\logic{KM}$ and $\logic{GL}$ admit filtration in this more general sense, but not in the standard sense. 

         Either version of stable canonical rules for msi rule systems can be used to generalize our proof of the Dummett-Lemmon conjecture, to be presented in \Cref{sec:dummettlemmon}, to the pair of signatures $\msimod$.  \label{nomsirule} \citet{Liao2023SCRfIML} was able to use our technique to prove a Dummett-Lemmon conjecture for rule systems which are quite similar to what we call msi rule systems, but where the intuitionistic modality satisfies at least the axioms of $\logic{S4}$. It does not seem, however, that anything of substance rests on the assumption of the $\logic{T}$ axiom. 
    \end{remark}

	If $\alg{H}, \alg{K}$ are (bi-)Heyting algebras, we call a map $h:\alg{H}\to \alg{K}$ \emph{stable} when $h$ is a bounded lattice homomorphism. Given $\heartsuit\in \{\to , \from\}$ and  $D^\heartsuit\subseteq H\times H$, we say that $h$ satisfies the $\heartsuit$-\emph{bounded domain condition}\footnote{The BDC$^\heartsuit$ was originally called \emph{closed domain condition} in, e.g., \cite{BezhanishviliEtAl2016SCR,BezhanishviliBezhanishvili2017LFRoHAaCF}, following \z's terminology for a similar notion in the theory of his canonical formulae.  The name \emph{stable domain condition} was later used in \cite{Bezhanishvili2023JFaATfIL} to stress the difference with \z's notion. However, this choice may create confusion between the BDC and the property of being a stable map. The terminology used in this paper is meant to avoid this, while concurrently highlighting the similarity between the geometric version of the BDC, to be presented in a few paragraphs, and the definition of a bounded morphism.} (BDC$^\heartsuit$) for $D^\heartsuit$ if 
	\[h(a\heartsuit b)=h(a)\heartsuit h(b)\] 
	for every $(a, b)\in D^\heartsuit$.
	It is not difficult to check that every stable map $h:\alg{H}\to \alg{K}$ satisfies $h(a\to b)\leq h(a)\to h(b)$ for every $(a, b)\in H$. If $\alg{H}\in \op{biHA}$, we also have $h(a\from b)\geq h(a)\from h(b)$ for every $(a, b)\in H$.

	Similarly, if $\alg{M, N}$ are modal (resp.\ tense) algebras, we call a map $h:\alg{M}\to \alg{N}$ \emph{stable} when $h$ is a Boolean algebra homomorphism which, moreover, satisfies
	\begin{equation*}
		h(\square a)\leq \square h(a) \qquad
		\left( \pastlog h(a)\leq h(\pastlog a) \right)
	\end{equation*}
	for each $a\in A$. Given $\heartsuit\in \{\square, \pastlog\}$ and  $D^\heartsuit \subseteq A$, we say that $h$ satisfies the $\heartsuit$-bounded domain condition (BDC$^\heartsuit$) for  $D^\heartsuit$ if 
	\[h(\heartsuit a)=\heartsuit h(a)\]
	for each $a\in D^\heartsuit$. In both the si/bsi and the modal/tense case, we say that $h$ satisfies the BDC for $ D$ if $h$ satisfies the BDC$^\heartsuit$ for $D^\heartsuit$ for each $D^\heartsuit$ in $D$. 
	
	The next result gives a uniform description of the refutation conditions of stable canonical rules on algebras in both the signatures under discussion. 
	\begin{proposition}[Cf. {\cite[Lemma 4.3]{BezhanishviliBezhanishvili2017LFRoHAaCF}, \cite[Thm. 5.4]{BezhanishviliEtAl2016SCR}}]
		For every stable canonical rule $\scr{A}{D}$ and every algebra $\alg{B}$ having the same signature as $\alg{A}$, we have that $\alg{B}\not\models \scr{A}{D}$ iff there is a stable embedding $h:\alg{B}\to \alg{A}$ satisfying the BDC for $D$. \label{refutalg}
	\end{proposition}
	\begin{proof}[Proof sketch]
		We use the identity $V(p_a)=h(a)$ to define either the desired stable embedding satisfying the BDC or the desired valuation. 
	\end{proof}

	Stable canonical rules also have uniform refutation conditions on spaces and Kripke frames. If $\spa{X, Y}$ are spaces, a map $f:\spa{X}\to \spa{Y}$ is called \emph{stable} when it is continuous and relation preserving, in the sense that $x\leq y$ implies $f(x)\leq f(y)$ for each $x, y\in X$. If $\spa{X, Y}$ are Kripke frames rather than spaces, we call $f:\spa{X}\to \spa{Y}$ \emph{stable} when it is relation preserving. In either case, given $\spa{d}\subseteq Y$, we say that $f$ satisfies the \emph{upward bounded domain condition} (BDC$\Uparrow$) for $\spa{d}$ when, for all $x\in X$, we have 
	\[{{\Uparrow}}f(x)\cap \spa{d}\neq\varnothing\Rightarrow f[{{\Uparrow}}x]\cap \spa{d}\neq\varnothing.\]
	This is to say: if there is $y\in \spa{d}$ such that $f(x)\preceq y$, then there must be some $z\in X$ with $x\preceq z$ and $f(z)\in \spa{d}$.  Analogously, we say that $f$ satisfies the \emph{downward bounded domain condition} (BDC$\Downarrow$) for $\spa{d}$ when 
	for all $x\in X$, we have 
	\[{{\Downarrow}}f(x)\cap \spa{d}\neq\varnothing\Rightarrow f[{{\Downarrow}}x]\cap \spa{d}\neq\varnothing.\]
	Thus the BDC$\Uparrow$ and BDC$\Downarrow$ are generalizations of the defining order-theoretic conditions of a bounded morphism.

	Given $\spa{D}^*\subseteq \wp(Y)$ for $*\in\{\Uparrow, \Downarrow\}$, we say that $f$ satisfies the BDC$^*$  for $\spa{D}^*$ when it satisfies the BDC$^*$  for each  $\spa{d}\in\spa{D}^*$. Given a tuple $\spa{D}$ with either one or two coordinates, we say that $f$ satisfies the BDC for $\spa{D}$ when $f$ satisfies the BDC$\Uparrow$ for the first coordinate of $\spa{D}$ and the BDC$\Downarrow$ for the second coordinate of $\spa{D}$, if it exists. Thus the BDC$\Uparrow$ is associated with the connectives $\square$ and $\to$, whereas the BDC$\Downarrow$ is associated with the connectives $\pastlog$ and $\from$. 

	Let $\scr{A}{D}$ be a stable canonical rule. We define a mapping $D^\heartsuit \mapsto  \spa{d}^\heartsuit$ by putting 
	$\spa{d}^\heartsuit:=\{\spa{d}^\heartsuit_{d}:d\in D\}$, with 
	\begin{align*}
		\spa{d}^\heartsuit_{(a, b)}&:=\beta(a)\smallsetminus \beta(b) &\heartsuit\in \{\to, \from\}\\
		\spa{d}^\square_a&:= -\beta(a) &\\
		\spa{d}^\pastlog_a&:= \beta(a), &
	\end{align*}
	where $\beta$ is the Stone map. We then let  $\spa{D}:=(\spa{d}^\heartsuit)_{\heartsuit\in \ops{\alg{A}}}$.

	\begin{proposition}
		For every stable canonical rule $\scr{A}{D}$ and for every space (resp.\ Kripke frame) $\spa{X}$, we have $\spa{X}\nvDash\scr{A}{D}$ iff there is a stable surjection $f:\spa{X}\to \alg{A}_*$ satisfying the BDC for $\spa{D}$. \label{refutspace}
	\end{proposition}
	\begin{proof}
		The case for modal spaces is proved in \cite[Thm. 3.6]{BezhanishviliEtAl2016SCR}; essentially the same argument can be used for the remaining cases involving spaces. In each of these cases, one appeals to the appropriate duality result from \Cref{topodual}. To establish the cases involving Kripke frames, one can adapt the same argument, but replacing appeals to duality results from \Cref{topodual} with appeals to duality results from \Cref{kripkedual}.
	\end{proof}
    \begin{convention}
        In view of \Cref{refutspace}, we adopt the convention of writing a stable canonical rule $\scr{A}{D}$ as $\scr{A_*}{\spa{D}}$ when working with spaces. 
    \end{convention}

	 Stable maps and the BDC are closely related to the filtration construction. We recall its definition in an algebraic setting, and state the fundamental theorem used in most of its applications. 
	\begin{definition}
		Let $\alg{A}$ be an algebra, $V$ a valuation on $\alg{A}$, and $\Theta$ a finite, subformula closed set of formulae. A (finite) model $(\alg{B}, V')$ is called a (\emph{finite}) \emph{filtration of $(\alg{A}, V)$ through $\Theta$} if the following conditions hold:
		\begin{enumerate}
			\item \emph{Heyting and bi-Heyting case}:
			\begin{enumerate}
				\item The bounded lattice reduct of $\alg{B}$ is isomorphic to the bounded sublattice of $\alg{A}$ generated by $V[\Theta]$;
				\item $V(p)=V'(p)$ for every propositional variable $p\in \Theta$;
				\item The inclusion $\subseteq:\alg{B}\to \alg{A}$ is a stable embedding satisfying the BDC$^\heartsuit$ for the set \[\{( V'(\varphi),  V'(\psi)):\varphi\heartsuit \psi\in \Theta\},\]
				for each $\heartsuit\in \ops{\alg{A}}$
			\end{enumerate}
			\item \emph{Modal and tense case}:
			\begin{enumerate}
				\item The Boolean algebra reduct of $\alg{B}$ is isomorphic to the Boolean subalgebra of $\alg{A}$ generated by $V[\Theta]$;
				\item $V(p)=V'(p)$ for every propositional variable $p\in \Theta$;
				\item The inclusion $\subseteq:\alg{B}\to \alg{A}$ is a stable embedding satisfying the BDC$^\heartsuit$ for the set \[\{V(\varphi):\heartsuit \varphi\in \Theta\},\]
				fore each $\heartsuit\in \ops{\alg{A}}$
			\end{enumerate}
		\end{enumerate}
	 \end{definition}
	 \begin{theorem}[Filtration theorem]
		Let $\alg{A}$ be an algebra, $V$ a valuation on $\alg{A}$, and $\Theta$ a finite, subformula closed set of formulae. If $(\alg{B}, V')$ is a filtration of $(\alg{A}, V)$ through $\Theta$, then for every $\varphi\in \Theta$ we have
	\[V(\varphi)=V'(\varphi).\]
	Consequently, for every rule $\Gamma/\Delta$ such that $\Gamma, \Delta\subseteq \Theta$  we have 
	\[\alg{A}, V\vDash \Gamma/\Delta\iff \alg{B}, V'\vDash \Gamma/\Delta.\]
\end{theorem}




The next result establishes that every rule is equivalent to finitely many stable canonical rules. The restriction of this lemma to si and modal rules was proved in \cite[Proposition 3.3]{BezhanishviliEtAl2016CSL}, \cite[Thm. 5.5]{BezhanishviliEtAl2016SCR}. 
\begin{lemma}[Cf. {\cite[Proposition 3.3]{BezhanishviliEtAl2016CSL}, \cite[Thm. 5.5]{BezhanishviliEtAl2016SCR}}]
	The following conditions hold:\label{rewrite} 
	\begin{enumerate}
		\item For every (b)si rule 
		$\Gamma/\Delta$ there is a finite set $\Xi$ of (b)si  stable canonical rules such that for any (bi-)Heyting algebra $\alg{K}$ we have  $\alg{K}\nvDash \Gamma/\Delta$ iff there is $\scr{H}{D}\in \Xi$ such that $\alg{K}\nvDash \scr{H}{D}$.
		\item For every modal rule $\Gamma/\Delta$ there is a finite set $\Xi$ of modal stable canonical rules of the form $\scrmod{M}{D}$ with $\alg{M}$ a $\logic{K4}$-algebra, such that for any $\logic{K4}$ algebra $\alg{N}$  we have that $\alg{N}\nvDash \Gamma/\Delta$ iff there is $\scrmod{M}{D}\in \Xi$ such that $\alg{N}\nvDash \scrmod{M}{D}$. 
        \item For every modal (resp.\ tense) rule $\Gamma/\Delta$ there is a finite set $\Xi$ of modal (resp.\ tense) stable canonical rules of the form $\scrmod{M}{D}$ with $\alg{M}$ a $\logic{S4(.t)}$-algebra, such that for any $\logic{S4(.t)}$ algebra $\alg{N}$  we have that $\alg{N}\nvDash \Gamma/\Delta$ iff there is $\scrmod{M}{D}\in \Xi$ such that $\alg{N}\nvDash \scrmod{M}{D}$. 
	\end{enumerate}
\end{lemma}
\begin{proof}
	We spell out the proofs of the si and bsi case to illustrate the exact role of filtration in the machinery of stable canonical rules.  When $\Gamma/\Delta$ is a rule, write $\mathit{Sfor}(\Gamma/\Delta)$ for the set of all formulas that are subformulas of some $\gamma\in \Gamma$ or some $\delta\in \Delta$. 
	Since bounded distributive lattices are locally finite, there are, up to isomorphism, only finitely many pairs $(\alg{H},  D)$ such that 
	\begin{itemize}
		\item $\alg{H}$ is a (bi-)Heyting  algebra which is at most $k$-generated as a bounded distributive lattice, where $k=|\mathit{Sfor}(\Gamma/\Delta)|$; 
		\item $ D=(D^\heartsuit)_{\heartsuit\in \ops{\alg{H}}}$ with $D^\heartsuit=\{(V(\varphi), V(\psi)):\varphi\heartsuit \psi\in \mathit{Sfor}(\Gamma/\Delta)\}$, where  $V$ is a valuation on $\alg{H}$ refuting $\Gamma/\Delta$.
	\end{itemize}
	Let $\Xi$ be the finite set of all rules $\scrsi{H}{D}$ for all such pairs $(\alg{H}, D)$, identified up to isomorphism. 
	
	$(\Rightarrow)$ Assume $\alg{K}\nvDash \Gamma/\Delta$ and take a valuation $V$ on $\alg{K}$ refuting $\Gamma/\Delta$. Consider the bounded distributive sublattice $\alg{J}$ of $\alg{K}$ generated by $V[\mathit{Sfor}(\Gamma/\Delta)]$. Since bounded distributive lattices are locally finite, $\alg{J}$ is finite. Moreover $\alg{J}$ may be viewed as a Heyting or bi-Heyting algebra by defining one or both of the following operations on $\alg{J}$:
	\begin{align*}
		a\rightsquigarrow b&:=\bigvee\{c\in J: a\land b\leq c\}\\
		a\leftsquigarrow b&:=\bigwedge\{c\in J: a\leq b\lor c\}.
	\end{align*}
	Define a valuation $V'$ on $\alg{J}$ with $V'(p)=V(p)$ if $p\in \mathit{Sfor}(\Gamma/\Delta)$, $V'(p)$ arbitrary otherwise. 
	Since $\alg{J}$ is a sublattice of $\alg{K}$, the inclusion $\subseteq$ is a stable embedding. 
	\begin{itemize}
		\item Let $\varphi\to \psi\in \mathit{Sfor}(\Gamma/\Delta)$. Then $V'(\varphi)\to V'(\psi)\in J$. Since $\subseteq$ is a stable embedding we have $V'(\varphi)\rightsquigarrow V'(\psi)\leq V'(\varphi)\to V'(\psi)$. Conversely, by the definition of $\rightsquigarrow$ we find $V'(\varphi)\rightsquigarrow V'(\psi)\land V'(\varphi)\leq V'(\psi)$. By the properties of Heyting algebras it follows that $V'(\varphi)\rightsquigarrow V'(\psi)\leq  V'(\varphi)\to  V'(\psi)$. Thus, $V'(\varphi)\rightsquigarrow V'(\psi)=  V'(\varphi)\to  V'(\psi)$
		\item By analogous reasoning, whenever $\varphi\from \psi\in \mathit{Sfor}(\Gamma/\Delta)$ we have that $V'(\varphi)\leftsquigarrow V'(\psi)=V'(\varphi)\from V'(\psi)$. 
	\end{itemize}
	We have thus shown  that the model $(\alg{J}, V')$ is a filtration of the model $(\alg{K}, V)$ through $\mathit{Sfor}(\Gamma/\Delta)$, which implies $\alg{J}, V'\nvDash \Gamma/\Delta$.
	
	$(\Leftarrow)$ Assume that there is $\scrsi{H}{D}\in \Xi$ such that $\alg{K}\nvDash \scrsi{H}{D}$. Let $V$ be the valuation associated with $ D$ in the sense spelled out above. Then $\alg{H}, V\nvDash \Gamma/\Delta$. Moreover, $(\alg{H}, V)$ is a filtration of the model $(\alg{K}, V)$, so by the filtration theorem it follows that $\alg{K}, V\nvDash \Gamma/\Delta$.
\end{proof}

The proofs of the modal and tense cases of \Cref{rewrite} are analogous, appealing to the local finiteness of Boolean algebras instead of the local finiteness of bounded distributive lattices. While filtrations of models based on modal or tense algebras are not unique, they can always be constructed. Furthermore, a model based on a $\logic{K4}$-algebra always has a filtration which is itself based on a $\logic{K4}$-algebra, and a model based on a $\logic{S4(.t)}$-algebra always has a filtration which is itself based on a $\logic{S4(.t)}$-algebra. These observations allow one to prove the second and third parts of \Cref{rewrite} by essentially the same argument. 

As a consequence of \Cref{rewrite} we obtain uniform axiomatizations of si, bsi, modal and tense rule systems in terms of stable canonical rules. 
\begin{theorem}[{\cite[Cf.][Proposition 3.4]{BezhanishviliEtAl2016CSL}}]
	The following conditions hold: \label{axiomatisationsi}
	\begin{enumerate}
		\item Any si, bsi, modal and tense rule system  is axiom\-atisable, over the least rule system of the same kind, by some set of  stable canonical rules;
		\item Every modal rule system above $\logic{K4}$ is axiomatizable, over $\logic{K4}$, by a set of stable canonical rules based on $\logic{K4}$-algebras. 
		\item Every modal (resp.\ tense) rule system above $\logic{S4(.t)}$  is axiomatizable, over $\logic{S4(.t)}$, by a set of stable canonical rules based on $\logic{S4(.t)}$-algebras.
	\end{enumerate}
\end{theorem}

We close this section with a brief comparison of our  stable canonical rules with  Je\v{r}ábek-style Canonical Rules. Our bsi and tense stable canonical rules generalize si and modal stable canonical rules in a way that mirrors the intimate connection existing between Heyting and bi-Heyting algebras on the one hand, and modal and tense algebras on the other. Just like a bi-Heyting algebra is nothing but a Heyting algebra whose order-dual is also a Heyting algebra, so every bsi stable canonical rule is a sort of ``independent fusion" between two si stable canonical rules, whose associated Heyting algebras are order-dual to one another. Similarly for the tense case.

Je\v{r}ábek-style si and modal canonical rules (like \z-style si and modal canonical formulae), by contrast, do not generalize as smoothly to the bsi and tense case. Algebraically, a Je\v{r}ábek-style si canonical rule may be defined as follows (cf. \cite{BezhanishviliBezhanishvili2009AAAtCFIC,BezhanishviliEtAl2016CSL}). 
\begin{definition}
 	Let $\alg{H}\in \mathsf{HA}$ be finite and let $D\subseteq H$. The \emph{si canonical rule} of $(\alg{H}, D)$ is the rule $\zeta(\alg{H}, D)=\Gamma/\Delta$, where
 	\begin{align*}
 		\Gamma:=&\{p_0\leftrightarrow \bot\}\cup\\
 		&\{p_{a\land b}\leftrightarrow p_a\land p_b:a, b\in H\}\cup\{p_{a\to b}\leftrightarrow p_a\to p_b:a, b\in H\}\cup\\
 		&\{p_{a\lor b}\leftrightarrow p_a\lor p_b:(a, b)\in D\}\\
 		\Delta:=&\{p_a\leftrightarrow p_b:a, b\in H\text{ with }a\neq b\}.
 	\end{align*}
\end{definition}
Generalizing the proof of \cite[Corollary 5.10]{BezhanishviliEtAl2016CSL}, one can show that every si rule is equivalent to finitely many si canonical rules. The key ingredient in this proof is a characterization of the refutation conditions for si canonical rules: $\zeta(\alg{H}, D)$ is refuted by a Heyting algebra $\alg{K}$ iff there is a $(\land, \to, 0)$-embedding $h:\alg{H}\to \alg{K}$ preserving $\lor$ on elements from $D$. Because $(\land, \to, 0)$-algebras are locally finite, a result known as \emph{Diego's theorem}, one can then reason as in the proof of \Cref{rewrite} to reach the desired result. 

It should be clear that if one defined the bsi canonical rule $\zeta_B(\alg{H}, D, D')$ by combining the rules  $\zeta(\alg{H}, D)$ and $\zeta({\alg{H}}, D')$ the same way bsi stable canonical rule combine si stable canonical rules, then $\zeta_B(\alg{H}, D, D')$ would be refuted by a bi-Heyting algebra $\alg{K}$ iff there is a bi-Heyting algebra embedding $h:\alg{H}\to \alg{K}$. Since the variety of bi-Heyting algebras is not locally finite, this refutation condition is clearly too strong to deliver a result to the effect that every bsi rule is equivalent to a set of bsi canonical rules. Without such a result, in turn, there is no hope of axiomatizing every rule system over $\logic{biIPC}$ by means of bsi canonical rules. 

Similar remarks hold in the tense case. \citet{BezhanishviliEtAl2011AAAtSLMC} show that the proof of the fact that every modal formula is equivalent, over $\logic{S4}$, to finitely many modal \z-style canonical formulae of $\logic{S4}$-algebras rests on an application of Diego's theorem \cite[cf.][Main Lemma]{BezhanishviliEtAl2011AAAtSLMC}. This has to do with how selective filtrations of $\logic{S4}$-algebras are constructed. Given a $\logic{S4}$-algebra $\alg{B}$ refuting a rule $\Gamma/\Delta$, a key step in constructing  a finite selective filtration of $\alg{B}$ through $\mathit{Sfor}(\Gamma/\Delta)$ consists in generating a $(\land, \to, 0)$-subalgebra of $\fragment \alg{A}$ from a finite subset of $O(A)$. This structure is guaranteed to be finite by Diego's theorem. On the most obvious ways of generalizing this construction to tense algebras, we would need to replace this step with one of the following:
\begin{enumerate}
	\item Generate both a $(\land, \to, 0)$-subalgebra of $\fragment \alg{A}$ and a  $(\lor, \from, 1)$-subalgebra of $\fragment \alg{A}$ from a finite subset of $O(A)$;
	\item Generate a bi-Heyting subalgebra of $\fragment \alg{A}$ from a finite subset of $O(A)$. 
\end{enumerate}
On option 1, Diego's theorem and its order dual would guarantee that both the $(\land, \to, 0)$-subalgebra of $\fragment \alg{A}$ and the  $(\lor, \from, 1)$-subalgebra of $\fragment \alg{A}$ are finite. However, it is not clear how one could then combine the two subalgebras into a bi-Heyting algebra, which is required to obtain a selective filtration based on a tense algebra. On option 2, on the other hand, we would indeed obtain a bi-Heyting subalgebra of $\fragment \alg{A}$, but not necessarily a finite one, since bi-Heyting algebras are not locally finite. 

We realize that the argument sketches just presented are far from conclusive, so we do not go as far as ruling out the possibility that Je\v{r}ábek-style bsi and tense canonical rules could  somehow be developed in such a way as to be a suitable tools for developing the theory of tense companions of bsi-rule systems. What such rules would look like, and in what sense they would constitute genuine generalizations of Je\v{r}ábek's canonical rules and \z's canonical formulae are interesting questions, but one we cannot hope to adequately answer here. The point we wish to stress is that answering this sort of questions  is a non-trivial matter, whereas generalizing stable canonical rules to the bsi and tense setting is a completely routine task. On our approach, exactly the same methods used in the si and modal case work equally well in the bsi and tense case.

\section{Blok-Esakia Theorems and the Kuznetsov-Muravitsky Isomorphism}

	\label{sec:mc}
	We now set out to develop the theory of modal and tense companions of si and bsi rule systems using the machinery of stable canonical rules just presented. For each of the three pairs of signatures $\simod, \bsiten$ and  $\msimod$ under discussion, we prove that the companions of a rule system form an interval, and establish a Blok-Esakia like result. The original Blok-Esakia theorem and the Kuznetsov-Muravitsky isomorphism follow as corollaries.

	\subsection{Novel Proofs}

	The main problem one needs to deal with in order to prove the results just announced is showing that each syntactic mapping $\greatest[\sig]$ is surjective on the codomain $\class{C}^+_\sig$ (recall we are using the notation introduced in \Cref{signatureconvention}.) The novelty of our approach lies in the use of stable canonical rules to establish that result. 

	Our strategy is centered on the following lemma. 
	\begin{lemma}[Main Lemma]
		Given $\sig\in \{\simod, \bsiten, \msimod\}$, let $\alg{M}\in \class{C}^+_\sig$. Then for every rule $\Gamma/\Delta$ in the classical signature of $\sig$ we have that $\alg{M}\models\Gamma/\Delta$ iff $\greatest[\sig]\fragment[\sig]\alg{M}\models \Gamma/\Delta$.\label{mainlemma}
	\end{lemma}
	In each of the three cases, the $(\Rightarrow)$ direction is immediate from \Cref{cor:representationHAS4}. We give full proofs of the other direction in the cases $\sig=\bsiten$ and $\sig=\msimod$. A proof of the case $\sig=\simod$ can be derived from the proof of the case $\sig=\bsiten$ by making minimal adaptations, which we shall sketch.
	\begin{proof}[Proof of Case $\sig=\bsiten$]
	    We prove the dual statement that $\alg{M}_*\nvDash \Gamma/\Delta$ implies $\greatest\fragment\alg{M}_*\nvDash \Gamma/\Delta$. Let $\spa{X}:=\alg{M}_*$. In view of \Cref{axiomatisationsi} it is enough to consider the case $\Gamma/\Delta=\scrmod{F}{\spa{D}}$, for $\spa{F}$ the dual of a finite $\logic{S4(.t)}$-algebra. So suppose $\spa{X}\nvDash \scrmod{F}{\spa{D}}$. By \Cref{refutspace}, there is a stable map $f:\spa{X}\to \spa{F}$ satisfying the BDC for $\spa{ D}:=(\spa{D}^\Uparrow, \spa{D}^\Downarrow)$.  We construct a stable map $g:\greatest\fragment\spa{X}\to \spa{F}$ which satisfies the BDC for $\spa{ D}$.
		
		Let $C:=\{x_1, \ldots, x_n\}\subseteq F$ be some cluster and let $Z_C:=f^{-1}(C)$. Since $f$ is relation-preserving, $Z_C$ does not cut clusters. Therefore, by \Cref{rhoproperties}, $\varrho[Z_C]$ is clopen, and so is $f^{-1}(x_i)$ for each $x_i\in C$. Now for each $x_i\in C$ let 
		\begin{align*}
			M_i&:=\mathit{max}_R(f^{-1}(x_i))\\
			N_i&:=\mathit{min}_R(f^{-1}(x_i)).
		\end{align*}
		By \Cref{propk4} and \Cref{grz-max}, both $M_i, N_i$ are closed, and moreover neither cuts any cluster. Consequently, by \Cref{rhoproperties} again, both $\varrho[M_i], \varrho[N_i]$ are closed as well.

		For each $x_i\in C$ let $O_i:=M_i\cup N_i$. Clearly, $O_i\cap O_j=\varnothing$ for each distinct $i, j\leq n$, and since no $O_i$ cuts any cluster this implies $\varrho[O_i]\cap \varrho[O_j]$ for each distinct $i, j\leq n$. We shall now  find disjoint clopens $U_1, \ldots, U_n\in \mathsf{Clop}(\greatest\fragment\spa{X})$ with $\varrho[O_i]\subseteq U_i$ and $\bigcup_i U_i=\varrho[Z_C]$. Let $k\leq n$ and assume that  $U_i$ has been defined for all $i<k$. If $k=n$, put $U_n=\varrho[Z_C]\smallsetminus\left(\bigcup_{i< k} U_i\right)$ and we are done. Otherwise set $V_k:=\varrho[Z_C]\smallsetminus\left(\bigcup_{i< k} U_i\right)$ and observe that it contains each $\varrho[O_i]$ for  $k\leq i\leq n$. By the separation properties of Stone spaces, for each $i$ with  $k<i\leq n$ there is some $U_{k_i}\in \mathsf{Clop}(\greatest\fragment\spa{X})$ with $\varrho[O_k]\subseteq U_{k_i}$ and $\varrho[M_i]\cap U_{k_i}=\varnothing$. Then set $U_k:=\bigcap_{k<i\leq n} U_{k_i}\cap V_k$. 
		
		We can now define a map 
		\begin{align*}
			g_C&: \varrho[Z_C]\to C\\
			z&\mapsto x_i\iff z\in U_i.
		\end{align*}
		Clearly, $g_C$ is relation preserving. Finally, define $g: \greatest\fragment\spa{X}\to F$ by setting
		\[
		g(\varrho(z)):=\begin{cases}
			f(z)&\text{ if } f(z)\text{ does not belong to any proper cluster }\\
			g_C(\varrho(z))&\text{ if }f(z)\in C\text{ for some proper cluster }C\subseteq F.
		\end{cases}
		\]
		Now, $g$ is evidently relation preserving. Moreover, it is continuous because both $f$ and each $g_C$ are. Thus, $g$ is a stable map. 
		
		We must now show that $g$ satisfies the BDC for $\spa{ D}$.  Suppose $Rg(\varrho(x)) y$ and $y\in \spa{d}$ for some $\spa{d}\in \spa{D}^\Uparrow$. By construction, $f(x)$ belongs to the same cluster as $g(\varrho(x))$, so also $Rf(x)y$. Since $f$ satisfies the BDC$^\Uparrow$ for $\spa{D}^\Uparrow$, there must be some $z\in X$ such that $Rxz$ and $f(z)\in \spa{d}$. Since $f^{-1}(f(z))\in \mathsf{Clop}(\spa{X})$, by \Cref{propk4} and \Cref{grz-max} there is $z'\in \mathit{max}(f^{-1}(f(z)))$ with $Rzz'$. 
		Then also $Rxz'$ and $f(z')\in \spa{d}$.
	   But from  $z'\in \mathit{max}(f^{-1}(f(z)))$ it follows that $f(z')=g(\varrho(z'))$ by construction, so we have $g(\varrho(z'))\in \spa{d}$. As clearly $R\varrho(x)\varrho(z')$, we have shown that $g$ satisfies the BDC$^\Uparrow$ for $\spa{D}^\Uparrow$. Analogous reasoning establishes that $g$ satisfies the BDC$^\Downarrow$ for $\spa{D}^\Downarrow$. By \Cref{refutspace} this implies $\greatest\fragment\spa{X}\not\models \scrmod{F}{\spa{D}}$.
	\end{proof} 

    \begin{proof}[Proof of Case $\sig=\simod$]
        To obtain a proof of this case, one runs essentially the same argument, but ignoring the sets $N_i$ in the construction of the map $g$. That is, one partitions each $\varrho[Z_C]$ into disjoint clopens, in such a way that each closed set of the form $\varrho[M_i]$, instead of $\varrho[O_i]$, is contained in one of such clopens. The rest of the construction is carried out the same way. 
    \end{proof}
	  
        \begin{proof}[Proof of Case $\sig=\msimod$]
           As before, we prove the dual statement  $\dualspa{M}\not\models \Gamma/\Delta$ implies $\greatest\fragment\dualspa{M}\not\models \Gamma/\Delta$. Let $\spa{X}:=\alg{M}_*$. Using \Cref{axiomatisationsi}, wlog, we restrict attention to the case $\Gamma/\Delta=\scrmod{F}{\spa{D}}$, for $\spa{F}$ the dual of a finite $\logic{K4}$-algebra. By \Cref{refutspace}, we take  a stable map $f:\spa{X}\to \spa{F}$ satisfying the BDC for $\spa{ D}$. 
		
	   Given a cluster $C:=\{x_1, \ldots, x_n\}\subseteq F$ we let $Z_C:=f^{-1}(C)$. Then $Z_C$ does not cut clusters, so by \Cref{rhoproperties}, $\varrho[Z_C]$ is clopen, and so is $f^{-1}(x_i)$ for each $x_i\in C$. For each $x_i\in C$ we let 
	   \begin{align*}
		   M_i&:=\mathit{max}_R(f^{-1}(x_i)).
	   \end{align*}
	   By \Cref{propk4} and \Cref{propgl}, each $M_i$ is clopen and does not cut any cluster. Consequently, by \Cref{rhoproperties}, each  $\varrho[M_i]$ is clopen. 
   
		Note $M_i\cap M_j=\varnothing$ holds for each distinct $i, j\leq n$, and since no $M_i$ cuts clusters we have $\varrho[M_i]\cap \varrho[M_j]=\varnothing$ for each distinct $i, j\leq n$. Constructing the desired partition of  $\varrho[Z_C]$ into clopen sets is now simpler. First, for $k<n$ let $U_k=\varrho[M_k]$. Then let $U_n:=\varrho[Z_C]\smallsetminus (U_1\cup\cdots \cup U_n)$. Then $U_1, \ldots, U_n$ are clopen sets which partition $\varrho[Z_C]$, such that $\varrho[M_k]\subseteq U_k$ for each $1\leq k\leq n$. 
	   
	   We define our map $g:\greatest[\msimod]\fragment[\msimod]\spa{X}\to \spa{F}$ as before. First, we let
	   \begin{align*}
		   g_C&: \varrho[Z_C]\to C\\
		   z&\mapsto x_i\iff z\in U_i.
	   \end{align*}
	   Clearly, $g_C$ is relation preserving. Finally, define $g: \greatest[\msimod]\fragment[\msimod]\spa{X}\to F$ by setting
	   \[
	   g(\varrho(z)):=\begin{cases}
		   f(z)&\text{ if $f(z)$ belongs to no proper cluster}  \\
		   g_C(\varrho(z))&\text{ if }f(z)\in C\text{ for some proper cluster }C\subseteq F.
	   \end{cases}
	   \]
	   It is clear that $g$ is a stable map. We show that it satisfies the BDC for $\spa{D}$. 
   
	   Suppose $Rg(\varrho(x))y$ and $y\in \spa{d}$ for some $\spa{d}\in \spa{D}$. If $\upset{f(x)}$ belongs to no proper cluster, then $g(\varrho(x))=f(x)$, so $Rf(x)z$. If $f(x)$ belongs to a proper cluster, then $g(\varrho (x))$ belongs to the same proper cluster, so again $Rf(x)z$. Either way, $Rf(x)z$. Since $f$ satisfies the BDC for $\spa{d}$, there must be some $x\in Z$ such that $Rxz$ and $f(z)\in \spa{d}$. Now, $f^{-1}(f(z))\in \clop{X}$, so by \Cref{propk4} and \Cref{propgl}, one of the following conditions hold:
	   \begin{enumerate}
		   \item $z\in \mathit{max}(f^{-1}(f(z)))$;
		   \item There is $z'\in \mathit{max}(f^{-1}(f(z)))$ with $Rzz'$.
	   \end{enumerate}
	   Either way, there is $z'\in \mathit{max}(f^{-1}(f(z)))$ such that $Rxz'$. But by construction, since  $z'\in \mathit{max}(f^{-1}(f(z)))$,  we have $f(z')=g(\varrho(z'))$. Consequently,  $g(\varrho(z'))\in \spa{d}$. As clearly $R\varrho(x)\varrho(z')$, we have shown that $g$ satisfies the BDC for $\spa{D}$.  
        \end{proof}

	Our main lemma has the following key consequence. 
	\begin{theorem}[Skeletal generation theorem]
		Let $\sig\in \{\simod, \bsiten, \msimod\}$. Every universal class $\class{U}\in \class{C}^+_\sig$ is generated by its skeletal elements, i.e., $\class{U}=\greatest[\sig] \fragment[\sig]\class{U}$.\label{unigrzgeneratedskel}
	\end{theorem}
	\begin{proof}
		Since $\greatest \fragment\alg{M}$ is a subalgebra of $\alg{M}$ for each $\alg{M}\in \class{U}$ (\Cref{cor:representationHAS4}), surely $\greatest\fragment\class{U}\subseteq \class{U}$. Conversely, suppose $\class{U}\nvDash \Gamma/\Delta$. Then there is $\alg{M}\in \class{U}$ with $\alg{M}\nvDash \Gamma/\Delta$. By \Cref{mainlemma} it follows that $\greatest\fragment\alg{M}\nvDash\Gamma/\Delta$. This shows $\mathsf{ThR}(\greatest\fragment\class{U})\subseteq \mathsf{ThR}(\class{U})$, which is equivalent to $\class{U}\subseteq \greatest\fragment\class{U}$. Hence, indeed, $\class{U}=\greatest \fragment\class{U}$.
	\end{proof}
	\begin{remark}
		The restriction of  \Cref{unigrzgeneratedskel} to varieties of $\logic{Grz}$-algebras  plays an important role in the algebraic proof of the Blok-Esakia theorem for si and modal logics given by \citet{Blok1976VoIA}.  The generalization to universal classes of modal algebras is explicitly stated and proved in \cite[Lemma 4.4]{Stronkowski2018OtBETfUC} using a generalization of Blok's approach, although it also follows from \cite[Theorem 5.5]{Jerabek2009CR}.  Blok establishes the restricted version of \Cref{unigrzgeneratedskel} as a consequence of what is now known as the \emph{Blok lemma}. The proof of the Blok lemma is notoriously involved. By contrast, our techniques afford a direct and, we believe, semantically transparent proof of \Cref{unigrzgeneratedskel}. \label{remark:blok}
	\end{remark}

	\subsection{Main results}

	The main results of this section follow from \Cref{unigrzgeneratedskel} by routine arguments; we review them here for completeness. We begin by establishing that the syntactic mappings  $\least, \fragment, \greatest$ commute with $\op{Alg}(\cdot)$.
\begin{lemma}
	Given $\sig\in \{\simod, \bsiten, \msimod\}$, let $\logic{L}\in \lat{Ext}(\logic{I}_\sig)$ and let $\logic{M}\in \lat{NExt}(\logic{C}_\sig)$. The following conditions hold:\label{prop:mcmapscommute}
	\begin{align}
		\op{Alg}(\least[\sig]\logic{L})&=\least[\sig] \op{Alg}(\logic{L}) \label{prop:mcmapscommute1}\\
		\op{Alg}(\greatest[\sig]\logic{L})&=\greatest[\sig] \op{Alg}(\logic{L})\label{prop:mcmapscommute2}\\
		\op{Alg}(\fragment[\sig]\logic{M})&=\fragment[\sig] \op{Alg}(\logic{M})\label{prop:mcmapscommute3}
	\end{align}
\end{lemma}
\begin{proof}
	(\ref{prop:mcmapscommute1}) For every  $\alg{M}\in \class{C}_\sig$ we have $\alg{M}\in \op{Alg}(\least\logic{L})$ iff  $\alg{M}\models T(\Gamma/\Delta)$ for all $\Gamma/\Delta\in \logic{L}$ iff $\fragment\alg{M}\models \Gamma/\Delta$ for all $\Gamma/\Delta\in \logic{L}$ iff $\fragment\alg{M}\in \op{Alg}(\logic{L})$ iff $\alg{M}\in \least \op{Alg}(\logic{L})$.
	
	(\ref{prop:mcmapscommute2}) In view of \Cref{unigrzgeneratedskel} it suffices to show that $\op{Alg}(\greatest\logic{L})$ and $\greatest \op{Alg}(\logic{L})$ have the same skeletal elements. So let $\alg{M}=\greatest\fragment\alg{M}\in \greatest \op{Alg}(\logic{L})$.  Since $\greatest \op{Alg}(\logic{L})$ is generated by $\{\greatest\alg{H}:\alg{H}\in \op{Alg}(\logic{L})\}$ as a universal class, by \Cref{cor:representationHAS4} and \Cref{lem:gtskeleton} we have $\alg{M}\models T(\Gamma/\Delta)$ for every $\Gamma/\Delta\in \logic{L}$. But then $\alg{M}\in \op{Alg}(\greatest\logic{L})$. Conversely, assume $\alg{M}=\greatest\fragment\alg{M}\in \op{Alg}(\greatest\logic{L})$. Then $\alg{M}\models T(\Gamma/\Delta)$ for every $\Gamma/\Delta\in \logic{L}$. By \Cref{lem:gtskeleton} this is equivalent to $\fragment\alg{M}\in \op{Alg}(\logic{L})$, therefore $\greatest\fragment\alg{M}=\alg{M}\in \greatest\op{Alg}(\logic{L})$.

	(\ref{prop:mcmapscommute3}) Let  $\alg{H}\in \fragment \op{Alg}(\logic{M})$. Then $\alg{H}=\fragment \alg{M}$ for some $\alg{M}\in \op{Alg}(\logic{M})$. It follows that for every si rule $T(\Gamma/\Delta)\in \logic{M}$ we have $\alg{M}\models T(\Gamma/\Delta)$, and so by \Cref{lem:gtskeleton} in turn $\alg{H}\models\Gamma/\Delta$. Therefore indeed $\alg{H}\in \op{Alg}(\fragment\logic{M})$. Conversely, for all  rules $\Gamma/\Delta$, if $\fragment\op{Alg}(\logic{M})\models \Gamma/\Delta$, then by \Cref{lem:gtskeleton} $\op{Alg}(\logic{M})\models T(\Gamma/\Delta)$, hence $\Gamma/\Delta\in \fragment\logic{M}$. Thus $\op{ThR}(\fragment\op{Alg}(\logic{M}))\subseteq  \fragment\logic{M}$, and so $\op{Alg}(\fragment\logic{M})\subseteq \fragment\op{Alg}(\logic{M})$.
\end{proof}

The result just proved leads straightforwardly to the following, purely semantic characterization of companions. 
\begin{lemma}
	Given $\sig\in \{\simod, \bsiten, \msimod\}$, let $\logic{L}\in \lat{Ext}(\logic{I}_\sig)$ and let $\logic{M}\in \lat{NExt}(\logic{C}_\sig)$. Then $\logic{M}$ is a companion of $\logic{L}$ iff $\op{Alg}(\logic{L})=\fragment[\sig]\op{Alg}(\logic{M})$. \label{mcsemantic}
\end{lemma}
\begin{proof}
	$(\Rightarrow)$ Assume $\logic{M}$ is a companion of $\logic{L}$. Then we have $\logic{L}=\fragment\logic{M}$. By \Cref{prop:mcmapscommute} $\op{Alg}(\logic{L})=\fragment\op{Alg}(\logic{M})$.

	$(\Leftarrow)$ Assume that $\op{Alg}(\logic{L})=\fragment\op{Alg}(\logic{M})$. Therefore, by \Cref{cor:representationHAS4}, $\alg{H}\in \op{Alg}(\logic{L})$ implies $\greatest \alg{H}\in \op{Alg}(\logic{M})$. This implies that for every  rule $\Gamma/\Delta$, $\Gamma/\Delta\in \logic{L}$ iff $T(\Gamma/\Delta)\in \logic{M}$.
\end{proof}

We can now prove the main results of this section. The first result asserts that the companions  of a rule system form an interval. 
\begin{theorem}[Interval theorem]
	Given $\sig\in \{\simod, \bsiten, \msimod\}$, let $\logic{L}\in \lat{Ext}(\logic{I}_\sig)$.  The companions of $\logic{L}$ form an interval in $\lat{NExt}(\logic{C}_\sig)$, where the least and greatest companions are given by $\least[\sig]\logic{L}$ and $\greatest[\sig]\logic{L}$. \label{intervaltheorem}
\end{theorem}
\begin{proof}
	In view of \Cref{prop:mcmapscommute} it suffices to prove that $\logic{M}$ is a  companion of $\logic{L}$ iff $\greatest\op{Alg}(\logic{L})\subseteq \op{Alg}(\logic{M})\subseteq\least\op{Alg}(\logic{L})$. 
	
	($\Rightarrow$) Assume $\logic{M}$ is a modal companion of $\logic{L}$. Then by \Cref{mcsemantic} we have $\op{Alg}(\logic{L})=\fragment\op{Alg}(\logic{M})$, therefore it is clear that $\op{Alg}(\logic{M})\subseteq \least \op{Alg}(\logic{L})$. To see that $\greatest\op{Alg}(\logic{L})\subseteq \op{Alg}(\logic{M})$ it suffices to show that every skeletal algebra in $\greatest\op{Alg}(\logic{L})$ belongs to $\op{Alg}(\logic{M})$. So let $\alg{M}\cong\greatest\fragment\alg{M}\in \greatest\op{Alg}(\logic{L})$. Then $\fragment\alg{M}\in \op{Alg}(\logic{L})$ by \Cref{lem:gtskeleton}, so there must be $\alg{N}\in \op{Alg}(\logic{M})$ such that $\fragment\alg{N}\cong \fragment\alg{M}$. But this implies $\greatest \fragment\alg{N}\cong \greatest \fragment\alg{M}\cong \alg{M}$, and as universal classes are closed under subalgebras, by \Cref{cor:representationHAS4} we conclude $\alg{M}\in \op{Alg}(\logic{M})$.
	
	($\Leftarrow$) Assume $\greatest\op{Alg}(\logic{L})\subseteq \op{Alg}(\logic{M})\subseteq\least\op{Alg}(\logic{L})$. It is an immediate consequence of \Cref{cor:representationHAS4} that $\fragment\greatest \op{Alg}(\logic{L})=\op{Alg}(\logic{L})$, which gives us $\fragment\op{Alg}(\logic{M})\supseteq\op{Alg}(\logic{L})$. But by construction $\fragment\op{Alg}(\logic{M})=\fragment\least \op{Alg}(\logic{L})$, hence $\fragment\op{Alg}(\logic{M})\subseteq\op{Alg}(\logic{L})$. Therefore, indeed, $\fragment\op{Alg}(\logic{M})=\op{Alg}(\logic{L})$, so by \Cref{mcsemantic} we conclude that  $\logic{M}$ is a modal companion of $\logic{L}$.
\end{proof}

The second result is an analogue of the Blok-Esakia theorem. We use the qualifier ``general'' to indicate that the theorem applies uniformly to three different pairs of signatures.  
\begin{theorem}[General Blok-Esakia theorem]
	Let $\sig\in \{\simod, \bsiten, \msimod\}$. The mappings $\greatest[\sig]$ and the restriction of $\fragment$ to $\lat{NExt}(\logic{C}^+_\sig)$ are mutually inverse complete lattice isomorphisms between $\lat{Ext}(\logic{I_\sig})$ and $\lat{NExt}(\logic{C}^+_\sig)$.
\label{blokesakiagen}
\end{theorem}
\begin{proof}
	It is enough to show that the algebraic class operators $\greatest: \mathbf{Uni}(\class{I}_\sig)\to \mathbf{Uni}(\class{C}^+_\sig)$ and $\fragment:\mathbf{Uni}(\class{C}^+_\sig)\to \mathbf{Uni}(\class{I}_\sig)$ are complete lattice isomorphisms and mutual inverses. 
	  Both maps are evidently order preserving, and preservation of infinite joins is an easy consequence of \Cref{lem:gtskeleton}. Let $\class{U}\in \mathbf{Uni}(\class{C}^+_\sig)$. Then $\class{U}=\greatest\fragment\class{U}$ by \Cref{unigrzgeneratedskel}, so $\greatest$ is surjective and a left inverse of $\fragment$. Now let $\class{U}\in \mathbf{Uni}(\class{I}_\sig)$. It is an immediate consequence of \Cref{cor:representationHAS4} that $\fragment\greatest \class{U}=\class{U}$. Hence $\fragment$ is surjective and a left inverse of $\greatest$. Thus $\greatest$ and $\fragment$ are mutual inverses, and therefore must both be bijections. 	
\end{proof}

We note that \Cref{intervaltheorem} remains true when restricted to lattices of \emph{logics} only. This result, in the case of pair of signatures $\simod$, was established by \citet{MaksimovaRybakov1974ALoNML}; see also \cite[Sec. 9.6]{ChagrovZakharyaschev1997ML}. The same holds for \Cref{blokesakiagen}. Thus we obtain, as corollaries, the original Blok-Esakia theorem (case $\sig=\simod$), Wolter's \cite{Wolter1998OLwC} generalization thereof to bsi and tense logics (case $\sig=\bsiten$), as well as the original Kuznetsov-Muravitsky isomorphism (case $\sig=\msimod$).
\begin{corollary}[General Blok-Esakia theorem for logics]
	Let $\sig\in \{\simod, \bsiten, \msimod\}$. The restrictions of the mappings $\greatest[\sig]:\lat{Ext}(\logic{I_\sig})\to \lat{NExt}(\logic{C}^+_\sig)$ and $\fragment[\sig]: \lat{NExt}(\logic{C}^+_\sig)\to \lat{Ext}(\logic{I_\sig})$ to the lattices of logics $\lat{ExtL}(\logic{I_\sig})$ and $\lat{NExtL}(\logic{C}^+_\sig)$ are complete lattice isomorphisms and mutual inverses.  \label{blokesakia}
\end{corollary}
\begin{proof}
	By construction, $\greatest$ and $\fragment$ preserve the property of being a logic, because the Gödel translation of an assumption free, single-conclusion rule is always an assumption free, single-conclusion rule.  Since we already know that the restrictions of $\greatest[\sig]$ and $\fragment[\sig]$ to logics are mutual inverses, it suffices to show that $\greatest[\sig]$ commutes with meets and joins. This is obvious for joins, since the join of two logics is a logic. For meets, it is enough to show that $\greatest[\sig]$ commutes with $\op{Taut}$: for each rule system $\logic{L}\in \lat{Ext}(\logic{I_\sig})$ we have $\greatest[\sig]\op{Taut}(\logic{L})=\op{Taut}(\greatest[\sig]\logic{L})$. For if this condition holds, then we have:
    \begin{align*}
        \greatest[\sig](\logic{L}\otimes_{\lat{ExtL}(\logic{I_\sig})}\logic{L'})&= \greatest[\sig](\op{Taut}(\logic{L}\otimes_{\lat{Ext}(\logic{I_\sig})}\logic{L'})) &\text{By \Cref{rulesystemlogic}}\\
        &=\op{Taut}(\greatest[\sig](\logic{L}\otimes_{\lat{Ext}(\logic{I_\sig})}\logic{L'}))\\
        &= \op{Taut}(\greatest[\sig]\logic{L}\otimes_{\lat{Ext}(\logic{I_\sig})}\greatest[\sig]\logic{L'})&\text{By \Cref{blokesakiagen}}\\
        &=\greatest[\sig]\logic{L}\otimes_{\lat{ExtL}(\logic{I_\sig})}\greatest[\sig]\logic{L'}. &\text{By \Cref{rulesystemlogic}}
    \end{align*}
    Note that the reasoning is analogous when we consider infinite meets.

    The claim indeed holds. Since $\greatest[\sig]$ is order-preserving and $\op{Taut}(\logic{L})\subseteq \logic{L}$ we have $\greatest[\sig]\op{Taut}(\logic{L})\subseteq\greatest[\sig]\logic{L}$. 
    Since $\op{Taut}$ is also order preserving and the left-hand side is already a logic, it follows that $\greatest[\sig]\op{Taut}(\logic{L})\subseteq\op{Taut}(\greatest[\sig]\logic{L})$.  Likewise, if $\logic{M}\in \lat{NExt}(\logic{I}_\sig^+)$, then $\fragment[\sig]\op{Taut}(\logic{M})\subseteq \op{Taut}(\fragment[\sig]\logic{M})$. Thus, 
    \begin{align*}
        \op{Taut}(\greatest[\sig]\logic{L})&=\greatest[\sig]\fragment[\sig]\op{Taut}(\greatest[\sig]\logic{L})\\
        &\subseteq\greatest[\sig]\op{Taut}(\fragment[\sig]\greatest[\sig]\logic{L})\\
        &=\greatest[\sig]\op{Taut}(\logic{L}),
    \end{align*}
    as desired.
    
\end{proof}


\section{Dummett-Lemmon Conjectures}
\label{sec:dummettlemmon}

In this last section we apply stable canonical rules to give an alternative proof of the \emph{Dummett-Lemmon conjecture} for rule systems. This result states that a (b)si rule system is Kripke complete iff its weakest modal companion is. We recall that  a rule system is called  \emph{Kripke complete} if it is of the form $\logic{L}=\{\Gamma/\Delta:\mathcal{K}\models \Gamma/\Delta\}$ for some class of Kripke frames $\mathcal{K}$. We also remind the reader that we will not be discussing msi rule systems in this section.


We will need to introduce and study new operations on stable canonical rules. We first define an operation taking a (b)si stable canonical rule to a modal (resp.\ tense) stable canonical rule  equivalent to the Gödel translation of the former. 
\begin{definition}
	Let $\scrsi{H}{D}$ be a (b)si stable canonical rule. The modal (resp.\ tense) stable canonical rule $\ruletrans{H}{D}$ is defined as the rule $\scrmod{\greatest H}{D_\circ}$, where ${D_\circ}:=( D_\circ^\heartsuit)_{\heartsuit\in \ops{\alg{H}}}$ and 
	\[ D_\circ^\square:=\{\neg a \lor b: (a, b)\in D^\to \}\qquad  D_\circ^\pastlog:=\{a\land \neg b: (a, b)\in D^\from\}.\]
\end{definition}
We call $\ruletrans{H}{D}$ the \emph{modalization} of $\scrsi{H}{D}$. Adopting our conventions for notating stable canonical rules using spaces rather than algebras, given a (b)si stable canonical rule $\scrsi{X}{D}$, the rule $\ruletrans{X}{\spa{D}}$ is just the rule $\scrmod{\greatest X}{\spa{D}}$. 

We call a rule $\scr{M}{D}$ \emph{modalized} when it is the modalization of some (b)si stable canonical rule. Dually, we may characterize modalized rules as follows. 
\begin{lemma}
    A modal (resp.\ tense) $\scrmod{F}{\spa{D}}$ stable canonical rule is modalized precisely when it satisfies the following conditions:\label{trans-dual}
    \begin{enumerate}
        \item $\spa{F}$ is partially ordered.
        \item Every $\spa{d}\in \spa{D}^\square$ is of the form $U\cap V$, where $U$ is an upset and $V$ is a downset. 
        \item If $\spa{D}^\pastlog$ is defined, then so is every  $\spa{d}\in\spa{D}^\pastlog$. \label{trans-dual-3}
    \end{enumerate}
\end{lemma}
\begin{proof}
    If $\scrmod{F}{\spa{D}}$ is modalized, then $\spa{F}$ is the dual of $\sigma \alg{H}$ for some finite (bi)Heyting algebra $\alg{H}$, so it is partially ordered. Furthermore, every $d\in D^\square$ is of the form $\neg a\lor b$ for $a, b\in H$, and so $\spa{d}$ is of the form $-\beta(\neg a\lor b)=\beta(a)\cap -\beta (b)$. But $\beta(a), \beta(b)$ are upsets, and the complement of an upset is a downset. By similar reasoning, we may infer \Cref{trans-dual-3}.

    Conversely, assume $\scrmod{F}{\spa{D}}$ satisfies the three conditions above. Then the dual of $\spa{F}$ is clearly of the form $\sigma \alg{H}$ for some finite (bi)Heyting algebra $\alg{H}$. Given $\spa{d}=U\cap V\in \spa{D}$, by (bi-)Esakia duality there must be $a, b\in H$ with $U=\beta (a)$ and $V=-\beta (b)$. But then $U\cap V=\beta (a)\cap -\beta(b)=-\beta(\neg a\lor b)$, and by definition $d=\neg a\lor b$. Likewise, if $\spa{D}^\pastlog$ is defined and $\spa{d}=U\cup V$ satisfies \Cref{trans-dual-3}, we find that $d= a\land \neg b$ for some $a, b\in H$. 
\end{proof}

We now verify that modalization indeed coincides, up to equivalence over $\logic{S4(.t)}$, with the Gödel translation. 
\begin{lemma}[Rule translation lemma]
	Let $\alg{M}\in \op{Alg}(\logic{S4(.t)})$. For any (b)si stable canonical rule $\scrsi{H}{D}$ we have \label{ruletrans}
	\[\alg{M}\models \ruletrans{H}{D}\iff \alg{M}\models T(\scrsi{H}{D}).\]
\end{lemma}
\begin{proof}
	Let $\spa{X}:=\dualspa{M}$ and $\spa{F}:=\dualspa{H}$. Then $\scrsi{H}{D}=\scrsi{F}{\spa{D}}$ and  $\ruletrans{H}{D}=\scrmod{\greatest F}{\spa{D}}$. 

	$(\Rightarrow)$ Suppose $\spa{X}\not\models T(\scrsi{F}{\spa{D}})$. Then, by \Cref{lem:gtskeleton}, $\fragment\spa{X}\not\models \scrsi{F}{\spa{D}}$. Consequently, there is a stable map $f: \fragment\spa{X}\to \spa{F}$ satisfying the BDC for $\spa{ D}$. We construct a stable map $g: \spa{X}\to \greatest\spa{F}$ that also satisfies the BDC for $\spa{ D}$. To this end, put \[g(x):= f(\varrho (x)).\]
	Now, $g$ is continuous because both $f$ and $\varrho$ are. Moreover, both $f$ and $\varrho$ are relation preserving, whence $g$ is as well. Thus $g$ is a stable map. We check that it satisfes the BDC for $\spa{ D}$. Let $\spa{d}\in \spa{D}^\Uparrow$ and $x\in X$. Suppose there is $y\in \spa{d}$ such that $R g(x) y$. Since $f$ satisfies the BDC$^\Uparrow$ for $\spa{d}$, there must be $\varrho(z)\in \fragment\spa{X}$ such that $\varrho(x)\leq \varrho(z)$ and $f(\varrho(z))=g(z)\in \spa{d}$. Moreover, since $\varrho$ is relation reflecting, we have $R xz$, showing $g$ satisfies the BDC$^\Uparrow$ for $\spa{D}^\Uparrow$. Similarly, $g$ satisfies the BCD$^\Downarrow$ for $\spa{D}^\Downarrow$. Consequently, $\spa{X}\not\models \scrmod{\greatest F}{\spa{D}}$.

	$(\Leftarrow)$ Suppose $\spa{X}\not\models \scrmod{\greatest F}{\spa{D}}$. Then there is a stable map $g:\spa{X}\to \greatest\spa{F}$ satisfying the BDC for $\spa{D}$. We construct a map $f:\fragment\spa{X}\to \spa{F}$ satisfying the BDC for $\spa{D}$. To this end, let 
	\[f(\varrho(x)):=g(x).\]
	Note that $f$ is well defined. For if $x$ and $y$ belong to the same cluster, we must have both $Rf(x)f(y)$ and $Rf(y)f(y)$. But $\spa{F}$ lacks proper clusters, showing $f(x)=f(y)$. Moreover, $f$ is relation preserving and continuous, hence a stable map.  It is relation preserving because $\varrho$ is relation reflecting and $g$ is relation preserving. To see that it is continuous, observe that $g^{-1}(U)$ never cuts clusters for any $U\subseteq F$, then apply \Cref{rhoproperties}. 

	We check that $f$ satisfies the BDC for $\spa{D}$. Let $\spa{d}\in \spa{D}^\Uparrow$ and $\varrho(x)\in \fragment\spa{X}$. Suppose there is $y\in \spa{d}$ such that $Rf\varrho(x)y$. Since $g$ satisfies the BDC$^\Uparrow$ for $\spa{d}$, there must be $z\in X$ such that $Rxz$ and $g(z)=f(\varrho(z))\in \spa{d}$. Since $\varrho$ is relation preserving, $\varrho(x)\leq\varrho(y)$, showing  $f$ satisfies the BCD$^\Uparrow$ for $\spa{D}^\Uparrow$. Similarly, $f$ satisfies the BCD$^\Downarrow$ for $\spa{D}^\Downarrow$. Consequently, $\fragment\spa{X}\not\models \scrsi{F}{\spa{D}}$, which by by \Cref{lem:gtskeleton} implies $\spa{X}\not\models T(\scrsi{F}{\spa D})$. 

    Now, consider the map $g:\spa{G}\to \spa{F}$ given by $g[x]=f(x)$. It is clearly well defined. We claim that $g$ is a stable surjection that satisfies the BDC for $\spa{D}$. Indeed, let $\spa{d}\in \spa{D}$ and $[x]\in Y$, and suppose that $\upset{g[x]}\cap \spa{d}\neq\varnothing$. By $g[x]=f(x)$ and the fact that $f$ satisfies the BDC for $\spa{D}$, it follows that there is $y\in X$ such that $Rxy$ and $g[y]=f(y)\in\spa{d}$, as desired. 
\end{proof}

We now show that every rule $\scrmod{F}{\spa{D}}$ where $\spa{F}$ is a $\logic{Grz(.t)}$-space may be equivalently rewritten as a finite conjunction of modalized stable canonical rules. First, some preliminary definitions. Let $\spa{X}$ be a finite $\logic{Grz}$-space. If $U\subseteq X$, let $\mathit{pas}(U)$ be the set of points that are passive in $U$.  The \emph{chunks} of a set $U\subseteq X$ are defined recursively as follows. We put $\chunk_1(U):=\mathit{pas}(U)$. Assuming $\chunk_i(U)$ has been defined, we put
\[\chunk_{i+1}(U):=\mathit{pas}\left(U\smallsetminus (\chunk_1(U)\cup\cdots\cup \chunk_i(U)\right)\]
whenever the right-hand side is non-empty; we leave $\chunk_{i+1}(U)$ undefined otherwise. Since $\spa{X}$ is finite, every $U\subseteq {X}$ only has finitely many chunks: we let the \emph{chunk height} of $U$ be the number of chunks it has. Moreover, observe that $\chunk_{i}(U)=\mathit{pas}(\chunk_{i}(U))$, for each $i$ less than or equal to the chunk height of $U$.

\begin{lemma}
	Let $\scrmod{M}{D}$ be a stable canonical rule with $\alg{M}\in \op{Alg}(\logic{Grz(.t)})$. Then there is a finite set $\Phi$ of modalized stable canonical rules, such that an $\logic{S4(.t)}$-algebra  $\alg{N}$ refutes  $\scrmod{M}{D}$ iff it refutes some $\scrmod{\greatest H}{E}\in \Phi$. \label{rewritetrans}
    
\end{lemma}
\begin{proof}
We prove the dual statement. To keep things simple, we only show the case of modal spaces; the case of tense spaces is an adaptation of the same argument. 

Let $\spa{F}$ be the dual of $\alg{M}$. Observe that there are, up to isomorphism, only finitely many pairs $(\spa{G}, \spa{E})$ satisfying the following conditions:
\begin{enumerate}
    \item $\spa{G}$ is a finite $\logic{Grz}$-space whose cardinality is at most $|F|\cdot 2^k$, where $k$ is the number of all chunks of any $\spa{d}\in \spa{D}$. 
    \item $\spa{E}=\{g^{-1}(\chunk_i(\spa{d})):\spa{d}\in \spa{D}\text{ and } i\text{ at most the chunk height of $\spa{d}$}\}$, where $g:\spa{G}\to \spa{F}$ is a stable surjection satisfying the BDC for $\spa{D}$. \label{map}
\end{enumerate}
We let $\Phi$ be the set of all  rules $\scrmod{G}{\spa{E}}$ for all such pairs $(\spa{G}, \spa{E})$. 

Note that each rule $\scrmod{G}{\spa{E}}$ is modalized. By definition, $\spa{G}$ is partially order. Moreover, if $g^{-1}(\chunk_i(\spa{d}))\in \spa{E}$, then \[g^{-1}(\chunk_i(\spa{d}))=\upset{g^{-1}(\chunk_i(\spa{d}))}\cap \downset{g^{-1}(\chunk_i(\spa{d}))}.\]
Indeed, if $x\in \upset{g^{-1}(\chunk_i(\spa{d}))}\cap \downset{g^{-1}(\chunk_i(\spa{d}))}$, then there are $y, z\in g^{-1}(\chunk_i(\spa{d}))$ such that  $Rzx$ and $Rxy$. Since $g$ is stable, it follows that $Rg(z)g(x)$ and $Rg(x)g(y)$. But since $\chunk_i(\spa{d})=\mathit{pas}(\chunk_i(\spa{d}))$, we must have $g(x)\in \chunk_i(\spa{d})$, else one could leave and re-enter $\chunk_i(\spa{d})$. Thus, by \Cref{trans-dual}, each $\scrmod{G}{\spa{E}}$ is modalized.

($\Rightarrow$) Let $\spa{X}$ be a $\logic{S4}$-space and suppose $\spa{X}\not\models \scrmod{G}{\spa E}$ for some $\scrmod{G}{\spa E}\in \Phi$. Then there is a stable surjection $f:\spa{X}\to \spa{G}$ satisfying the BDC for $\spa{E}$. Let $g:\spa{G}\to \spa{F}$ be the stable surjection satisfying the BDC for $\spa{D}$ given by \Cref{map} from the definition of $\Phi$. By definition, 
$\spa{E}=\{g^{-1}(\chunk_i(\spa{d})):\spa{d}\in \spa{D}, i\text{ at most the chunk height of $\spa{d}$}\}$.

Consider the map $g\circ f:\spa{X}\to \spa{F}$. We claim that $g\circ f$ is a stable surjection that satisfies the BDC for $\spa{D}$. That $g\circ f$ is surjective follows because both $f, g$ are. Likewise, $g\circ f$ is stable because both $f, g$ are. To check the BDC, take any $x\in X$ and suppose $\upset{g(f(x))}\cap \spa{d}\neq\varnothing$ for some $\spa{d}\in \spa{D}$. Since $g$ satisfies the BDC for $\spa{D}$, there must be some $f(y)\in \spa{G}$ such that $Rf(x)f(y)$ and $g(f(y))\in \spa{d}$. Let $\chunk_i(\spa{d})$ be the unique chunk of $\spa{d}$ such that $g(f(y))\in \chunk_i(\spa{d})$. Then $f(y)\in g^{-1}(\chunk_i(\spa{d}))$. By definition, $g^{-1}(\chunk_i(\spa{d}))\in \spa{E}$. Since $f$ satisfies the BDC for $\spa{E}$, there must be some $z\in X$ such that $Rxz$ and $f(z)\in g^{-1}(\chunk_i(\spa{d}))$. In other words, $g(f(z))\in \chunk_i(\spa{d})\subseteq \spa{d}$. This shows that, indeed, $g\circ f$ satisfies the BDC for $\spa{D}$. We may then conclude $\spa{X}\not\models \scrmod{F}{\spa{D}}$.

($\Leftarrow$) Let $\spa{X}$ be a $\logic{S4}$-space.  Assume $\spa{X}\not\models \scrmod{F}{\spa{D}}$. Then there is a stable surjection $f:\spa{X}\to \spa{F}$ that satisfies the BDC for $\spa{D}$. We  define an equivalence relation on $\spa{X}$ as follows. We put $x\sim y$ when both 
\begin{enumerate}[label=(\roman*), ref=\roman*]
    \item $f(x)=f(y)$, and \label{quotient1}
    \item for every $\chunk_i(\spa{d})$ with $\spa{d}\in \spa{D}$, we have \label{quotient2}
\[\upset{x}\cap f^{-1}(\chunk_i(\spa{d}))\neq\varnothing\iff\upset{y}\cap f^{-1}(\chunk_i(\spa{d}))\neq\varnothing.\]
\end{enumerate}
In other words, $x\sim y$ holds when $x$ and $y$ have the same image under $f$, and ``see'' the $f$-preimages of exactly the same chunks of domains from $\spa{D}$. We write $[x]$ for the equivalence class of $x$ under $\sim$. Next, we define a relation on equivalence classes of $\sim$. We put $R[x][y]$ when both
\begin{enumerate}[label=(\roman*),ref=\roman*,resume]
    \item $Rf(x)f(y)$, and \label{quotient3}
    \item for every $\chunk_i(\spa{d})$ with $\spa{d}\in \spa{D}$, we have  \label{quotient4}
    \[\upset{y}\cap f^{-1}(\chunk_i(\spa{d}))\neq\varnothing\Rightarrow\upset{x}\cap f^{-1}(\chunk_i(\spa{d}))\neq\varnothing.\]
\end{enumerate}
It should be clear that $R$ is well defined. We let $\spa{Y}$ denote the modal space that results from equipping the quotient of $\spa{X}$  under $\sim$ with $R$. 

Observe that, by definition, $\sim$ refines the partition whose cells are points with the same $f$-images. The cardinality of that partition is clearly $|F|$. But $\sim$ splits each cell of this partition into at most $2^k$ sub-cells,  where $k$ is the total number of chunks of any $\spa{d}\in \spa{D}$. Consequently, the cardinality of $\spa{G}$ is at most $|F|\cdot 2^k$, as required by the definition of $\Phi$. 

Furthermore, we claim that $\spa{G}$ is a $\logic{Grz}$-space. Since $\spa{G}$ is finite, we need only check that its relation $R$ is a partial order. $R$ is clearly reflexive and transitive. For antisymmetry, assume $R[x][y]$ and $R[y][x]$. By (\ref{quotient3}), we have $Rf(x)f(y)$ and $Rf(y)f(x)$, which implies $f(x)=f(y)$ because $\spa{F}$ is partially ordered. Moreover, by (\ref{quotient4}), we have that $\upset{y}\cap f^{-1}(\chunk_i(\spa{d}))\neq\varnothing$ holds exactly when $\upset{x}\cap f^{-1}(\chunk_i(\spa{d}))\neq\varnothing$ does, for each $\chunk_i(\spa{d})$ with $\spa{d}\in \spa{D}$. But then $[x],[y]$ meet conditions (\ref{quotient1}) and (\ref{quotient2}), showing $[x]=[y]$. 

Let us define a map $g: \spa{G}\to \spa{F}$ by putting $g[x]=f(x)$. Then $g$ is a stable surjection satisfying the BDC for $\spa{D}$. Indeed, $g$ is surjective because $f$ and the quotient map both are. Moreover, the way we defined the relation of $\spa{Y}$ immediately implies that $g$ is relation preserving, and continuity  follows from the finiteness of $\spa{G}$. For the BDC, suppose $\upset{g[x]}\cap \spa{d}\neq\varnothing$ for some $\spa{d}\in \spa{D}$. Then $\upset{f(x)}\cap\spa{d}\neq\varnothing$.  Since $f$ satisfies the BDC for $\spa{D}$, there must be $y\in X$ with $Rxy$ and $f(y)=g[y]\in \spa{d}$. Since $Rxy$ implies $R[x][y]$, this shows that $g$ satisfies the BDC for $\spa{D}$. 

Via $g$, we may then define, in accordance with the definition of $\Phi$, 
\[\spa{E}=\{g^{-1}(\chunk_i(\spa{d})):\spa{d}\in \spa{D}, i\text{ at most the chunk height of $\spa{d}$}\}.\]
It follows that $\scrmod{G}{\spa E}\in \Phi$. 

We show that $\spa{X}\not\models \scrmod{G}{\spa E}$, by showing that the quotient map  $x\mapsto [x]$ is a stable surjection satisfying the BDC for $\spa{E}$. The quotient map is clearly relation preserving and surjective. Moreover, it is continuous, because each equivalence class under $\sim$ is definable as a finite intersection of clopens. Thus, it is a stable surjection. Let us check the BDC. Let $x\in X$ and $g^{-1}(\chunk_i(\spa{d}))\in \spa{E}$, and suppose $\upset{[x]}\cap g^{-1}(\chunk_i(\spa{d}))\neq\varnothing$. This means that there is $[y]\in g^{-1}(\chunk_i(\spa{d}))$ such that $R[x][y]$. By the definition of $g$, this is to say $y\in f^{-1}(\chunk_i(\spa{d}))$. A fortiori, $\upset{y}\cap f^{-1}(\chunk_i(\spa{d}))\neq\varnothing$. By condition (iv) in the definition of $R$, we may then infer that $\upset{x}\cap f^{-1}(\chunk_i(\spa{d}))\neq\varnothing$. In other words, there must be $z\in X$ with $Rxz$ and $f(z)\in \chunk_i(\spa{d})$. But $Rxz$ implies $R[x][z]$, and $f(z)\in \chunk_i(\spa{d})$ is equivalent to $[z]\in g^{-1}(\chunk_i(\spa{d}))$, as desired. 
\end{proof}

\begin{remark}
    It is a straightforward consequence of the Blok-Esakia theorem and \Cref{ruletrans} that every modal (resp.\ tense) rule is equivalent over $\logic{Grz(.t)}$ to a set  of modalized stable canonical rules. Indeed, given a modal rule $\Gamma/\Delta$, the modal rule system $\logic{Grz(.t)}\oplus\Gamma/\Delta$ must be of the form $\sigma \logic{L}$, for some (b)si rule system $\logic{L}$. We know that $\logic{L}$ must be axiomatizable, over $\logic{(bi)IPC}$, by a set of (b)si stable canonical rules $\Psi$. But then $\sigma\logic{L}=\logic{Grz(.t)}\oplus \{\ruletrans{H}{D}: \scrsi{H}{D}\in \Psi\}$ by \Cref{ruletrans}, which is to say that $\Gamma/\Delta$ is equivalent, over $\logic{Grz(.t)}$, to $\{\ruletrans{H}{D}: \scrsi{H}{D}\in \Psi\}$. Furthermore, one direction of this equivalence remains true over $\logic{S4(.t)}$. Indeed,  $\rho (\logic{S4(.t)}\oplus \Gamma/\Delta)=\rho(\logic{Grz(.t)}\oplus \Gamma/\Delta)$, so by \Cref{ruletrans} we have $\tau\rho(\logic{S4(.t)}\oplus \Gamma/\Delta)=\logic{S4(.t)}\oplus \{\ruletrans{H}{D}: \scrsi{H}{D}\in \Psi\}\subseteq \logic{S4(.t)}\oplus \Gamma/\Delta$. This is to say $\Gamma/\Delta$ implies each $\ruletrans{H}{D}$ over $\logic{S4(.t)}$. 

    These observations do not imply \Cref{rewritetrans}: for all we have said, $\Psi$  might be infinite, and the above reasoning does not establish that both directions of the equivalence go through when restricting attention to rules based on $\logic{Grz(.t)}$-spaces. That being said, the observations in this remark would be enough to carry out our proof of the Dummett-Lemmon conjecture. This is the strategy followed by \cite{Liao2023SCRfIML} in a generalization of our technique. We chose to rely on \Cref{rewritetrans} because we find the construction it employs independently interesting. A similar construction can be used to establish that $\logic{Grz(.t)}$ admits filtration, albeit in a somewhat non-standard sense. See \cite[Thm. 2.74]{Cleani2021TEVSCR}. 
\end{remark}

fGrz The last notion we need to introduce is that of a \emph{collapsed} stable canonical rule. 
\begin{definition}
	Let $\scrmod{M}{D}$ be a stable canonical rule with $\alg{M}\in \op{Alg}(\logic{S4(.t)})$. The \emph{collapsed stable canonical rule} is defined as the rule $\scrmod{\greatest\fragment M}{\greatest\fragment D}$, where ${\greatest\fragment D}:=(\greatest \fragment D^\heartsuit)_{\heartsuit\in \ops{\alg{M}}}$ and 
	\[\greatest\fragment D^\heartsuit:=\left\{ \bigwedge\{b\in B(O(\alg{M})): a\leq b\} : a\in D^\heartsuit\right\}.\]
\end{definition}

To understand the intuition behind collapsed rules, it is helpful to characterize them dually. Observe that the mapping on $\alg{M}$ given by
\[a\mapsto \bigwedge\{b\in B(O(\alg{M})): a\leq b\}\]
is the algebraic dual of the cluster collapse map on $\dualspa{M}$, in the sense that 
\[\beta \left(\bigwedge\{b\in B(O(\alg{M})): a\leq b\}\right)= \varrho[\beta (a)].\]
Consequently, the collapsed rule $\scrmod{\greatest\fragment M}{\greatest\fragment D}$ is identical to the rule $\scrmod{\greatest \fragment \dualspa{M}}{\greatest\fragment\spa{D}}$, where $\greatest\fragment\spa{D}$ is obtained by setting 
\begin{align*}
	\greatest\fragment \mathfrak{D}^\heartsuit&:=\{\varrho [\mathfrak{d}]: \mathfrak{d}\in \mathfrak{D}^\heartsuit\}\qquad &\heartsuit\in \{\Uparrow, \Downarrow\}\\
	{\greatest\fragment\spa{D}}&:=(\greatest\fragment \spa{D}^\heartsuit)_{\spa{D}^\heartsuit\in\spa{ D}}. &
\end{align*}
In other words, $\scrmod{\greatest\fragment \dualspa{M}}{\greatest\fragment \mathfrak{D}}$ is obtained from $\scrmod{F}{\mathfrak{D}}$ by collapsing all clusters in $\mathfrak{F}$ and in the sets of domains $\mathfrak{D}^\heartsuit$ as well.

Collapsed rules obey the following refutation condition on spaces and Kripke frames.
\begin{lemma}[Rule Collapse Lemma]
	For all $\mathfrak{X}\in \mathsf{Spa}(\logic{S4(.t)})$ and any stable canonical rule $\scrmod{F}{\mathfrak{D}}$ such that $\mathfrak{F}\in \mathsf{Spa}(\logic{S4(.t)})$, if $\mathfrak{X}\nvDash \scrmod{F}{\mathfrak{D}}$, then $\greatest\fragment \mathfrak{X}\nvDash \scrmod{\greatest\fragment F}{\greatest\fragment \mathfrak{D}}$.\label{rulecollapse} Moreover, the same holds if $\mathfrak{X}$ is a reflexive and transitive Kripke frame.
\end{lemma}
\begin{proof}
	Assume $\mathfrak{X}\nvDash\scrmod{F}{\mathfrak{D}}$. Then there is a stable map $f:\mathfrak{X}\to \mathfrak{F}$ that satisfies the BDC for $\mathfrak{ D}$. Consider the map $g:\greatest\fragment\mathfrak{X}\to \greatest\fragment\mathfrak{F}$ given by
	\[g(\varrho(x))=\varrho(f(x)).\]
	Now $R\varrho(x)\varrho(y)$ implies $Rxy$, and since $f$ is relation preserving also $Rf(x)f(y)$, which implies $R\varrho(f(x)) \varrho(f(y))$. So $g$ is relation preserving. Furthermore, again because $f$ is relation preserving we have that for any $U\subseteq F$, the set $f^{-1}(U)$ does not cut clusters, whence $g^{-1}(U)=\varrho[f^{-1}(\varrho^{-1}(U))]$ is clopen for any $U\subseteq \varrho[F]$, as $\fragment\mathfrak{X}$ has the quotient topology. Thus, $g$ is continuous. 
	
	Let us check that $g$ satisfies the BDC for ${\greatest\fragment\mathfrak{D}}$. Let $\spa{d}\in \spa{D}^\Uparrow$ and suppose  that ${\Uparrow} g(\varrho(x))\cap \varrho[\mathfrak{d}]\neq \varnothing$. Then there is some $\varrho(y)\in \varrho[F]$ with $R\varrho(f(x))\varrho(y)$ and $\varrho(y)\in \varrho[\mathfrak{d}]$. By construction, wlog we may assume that $y\in \mathfrak{d}$. As $\varrho$ is relation reflecting it follows that $Rf(x) y$, and so we have that ${\Uparrow}[ f(x)]\cap \mathfrak{d}\neq\varnothing$. Since $f$ satisfies the BDC$^\Uparrow$ for $\mathfrak{D}$ we conclude that $f[{\Uparrow} x]\cap \mathfrak{d}\neq\varnothing$. So, there is some $z\in X$ with $Rxz$ and $f(z)\in \mathfrak{d}$. By definition, $\varrho(f(z))\in \varrho[\mathfrak{d}]$. Hence we have shown that $\varrho[f[{\Uparrow}x]]\cap \varrho[\mathfrak{d}]\neq\varnothing$, and so $g$ indeed satisfies the BDC$^\Uparrow$ for $\mathfrak{D}^\Uparrow$. Similarly, $g$ indeed satisfies the BDC$^\Downarrow$ for $\mathfrak{D}^\Downarrow$. The case where $\spa{X}$ is a Kripke frame is analogous. 
\end{proof}

 We are now ready to prove the Dummett-Lemmon conjecture for rule systems. 
\begin{theorem}[Dummett-Lemmon conjecture for rule systems]
	A (b)si rule system $\logic{L}$ is Kripke complete iff $\least\logic{L}$ is.\label{dummettlemmon}
\end{theorem}
\begin{proof}
	$(\Rightarrow)$ Let $\logic{L}$ be Kripke complete. Suppose that $\Gamma/\Delta\notin \least\logic{L}$. Then there is $\spa{X}\in \op{Spa}(\least\logic{L})$ such that $\spa{X}\nvDash \Gamma/\Delta$. By \Cref{rewrite}, we may assume, wlog, that $\Gamma/\Delta=\scrmod{F}{\mathfrak{D}}$ for $\mathfrak{F}$ a preorder. By the Rule Collapse lemma, it follows that $\greatest\fragment\spa{X}\not\models \scrmod{\greatest\fragment F}{\greatest\fragment \spa{D}}$. Let $\Phi$ be the set of modalized stable canonical rules whose conjunction is equivalent, over $\logic{S4(.t)}$, to $\scrmod{\greatest\fragment F}{\greatest\fragment \spa{D}}$, given by \Cref{rewritetrans}. Then, by \Cref{rewritetrans}, there is a modalized stable canonical rule $\scrmod{\greatest G}{\spa{E}}\in \Phi$ such that $\greatest \fragment\spa{X}\not\models \scrmod{\greatest G}{\spa{E}}$. By the Rule Translation Lemma, it follows that $\fragment\spa{X}\not\models \scrsi{G}{\spa{E}}$. By \Cref{mcsemantic} and the fact that $\spa{X}\in \op{Spa}(\least\logic{L})$ it follows that $\fragment\spa{X}\in \op{Spa}(\logic{L})$. Consequently, $\scrsi{G}{\spa{E}}\notin \logic{L}$. Since $\logic{L}$ is Kripke complete, there must be a (b)si Kripke frame $\spa{Y}$ such that $\spa{Y}\not\models \scrsi{G}{\spa{E}}$. Therefore, by the Rule Translation Lemma again, $\greatest \spa{Y}\not\models \scrmod{\greatest G}{\spa{E}}$. Since $\greatest\spa{Y}$ is an $\logic{S4}$-Kripke frame, by \Cref{rewritetrans} and \Cref{kripkedual} it follows that $\greatest\spa{Y}\not\models \scrmod{\greatest \fragment F}{\greatest \fragment \spa{D}}$. Thus, there is a stable map $f:\greatest\spa{Y}\to \greatest \fragment \spa{F}$ satisfying the BDC for $\greatest\fragment \spa{D}$. 

	The goal now is to construct, from $\greatest\spa{Y}$ and $f$ respectively, a Kripke frame $\spa{Z}$ for $\logic{\least L}$ and a stable map $g:\spa{Z}\to \spa{F}$ satisfying the BDC for $\spa{D}$. We do so as follows. For each $x\in \varrho[F]$, enumerate $\varrho^{-1}(x):=\{x_1, \ldots, x_{k_x}\}$. Working in $\greatest\spa{Y}$, for each $y\in f^{-1}(x)$, replace $y$ with a $k_x$-cluster $y_1, \ldots, y_{k_x}$, and extend the relation $R$ clusterwise: $Ry_iz_j$ iff either $y=z$ or $Ryz$. This constitutes our Kripke frame $\spa{Z}$. Note that $\spa{Z}\models \logic{\least L}$, because $\fragment \spa{Z}\cong \spa{Y}$ (\Cref{mcsemantic}). For convenience, we identify $\fragment\spa{Z}$ and $\spa{Y}$. For every $x\in \varrho [F]$ define a map $g_x:f^{-1}(x)\to \varrho^{-1}(x)$ by setting $g_x(y_i)=x_i$ ($i\leq k_x$). Finally, define $g:\spa{Z}\to \spa{F}$ by putting $g=\bigcup_{x\in \varrho[F]} g_x$.

	The map $g$ is evidently well defined, surjective, and relation preserving. We claim that moreover, it satisfies the BDC for $\spa{D}$. To see this, let $\spa{d}\in \spa{D}^\Uparrow$ and suppose that ${\Uparrow} g(y_i) \cap \spa{d}\neq\varnothing$. Then there is $x_j\in F$ with $x_j\in \spa{d}$ and $Rg(y_i)x_j$. By construction also $\varrho(x_j)\in \varrho [\spa{d}]$ and $R f(\varrho (y_i))\varrho(x_j)$. As $f$ satisfies the BDC$^\Uparrow$ for $\greatest\fragment\spa{D}^\Uparrow$ it follows that there is some $z\in Y$ such that $R\varrho(y_i)z$ and $f(z)\in \varrho[\spa{d}]$. We may view $z$ as $\varrho(z_n)$ where $\varrho^{-1}(f(z))$ has cardinality $k\geq n$. Surely $Ry_iz_n$. Furthermore, since $f(z)\in \varrho[\spa{d}]$ there must be some $m\leq k$ such that $f(z)_m=g(z_m)\in \spa{d}$. By construction $Rz_nz_m$ and so in turn $Ry_iz_m$. This establishes that $g$ indeed satisfies the BDC$^\Uparrow$ for $\spa{D}^\Uparrow$. Analogous reasoning shows that $g$ satisfies the BDC$^\Downarrow$ for $\spa{D}^\Downarrow$.  Thus we have shown $\spa{Z}\nvDash \scrmod{F}{\spa{D}}$. Since $\spa{Z}\models \least\logic{L}$, it follows that $\least \logic{L}$ is Kripke complete.
\end{proof}

\section{Conclusion and Further Work}
\label{sec:conclusion}

This paper presented a novel approach to the study of modal companions and related notions based on stable canonical rules. We hope to have shown that our method is effective and quite uniform. With only minor adaptations to a fixed collection of techniques, we provided a unified treatment of the theories of modal and tense companions, and of the Kuznetsov-Muravitsky isomorphism. We both offered alternative proofs of classic theorems and established new results. 


The techniques presented in this paper are based on a blueprint easily applicable across signatures. Stable canonical rules can be formulated for any class of algebras which admits a locally finite expandable reduct in the sense of \cite[Ch. 5]{Ilin2018FRLoSNCL}, and once stable canonical rules are available there is a clear recipe for adapting our strategy to the case at hand. We propose that further research be done in this direction, in particular addressing the following topics.

Firstly, there are several more general notions of a msi rule system than that we have been working on, and one could try and study the theory of modal companions of such msi rule systems using our method. Some work in this direction has already been done. \cite{Liao2023SCRfIML} uses our methods to study \emph{bimodal} companions of rule systems over $\logic{IPC}\otimes\logic{S4}$. But there are more general settings to consider. For example, one can try replacing $\logic{S4}$ to a weaker modal logic, or consider systems in a richer signature with a primitive possibility operator. \cite{WolterZakharyaschev1997IMLAFoCBL,WolterZakharyaschev1997OtRbIaCML} give a very general definition of a msi logics, subsuming the cases we just mentioned, and study the theory of their polymodal companions. We conjecture that our techniques can recover several of the main known results in this area and generalize them to rule systems. 

A second avenue for further research is the theory of modal companions of  extensions of the \emph{Heyting-Lewis logic}, which expands superintuitionistic logic with a strict implication connective. Early work in this area began with \citet{GrootEtAl2021GMTaBEfHLI}, and \cite{Liao2023SCRfIML} more recently applied our methods to this setting. However, several results remain open, including whether an analogue of the Blok-Esakia theorem holds.


\vspace{4mm}

\noindent
{\bf Acknowledgment.} We are very grateful to the referee for many useful comments that led to a substantial restructuring and shortening of the paper. The referee also spotted an error in one of the original proofs. We are also thankful to Rodrigo Almeida, Guram Bezhanishvili, Cheng Liao, and Tommaso Moraschini for helpful comments and interesting discussions. 

\vspace{4mm}

\noindent
{\bf Funding.} The second author acknowledges the support of the USC Dornsife Graduate School.


	\bibliographystyle{default}
	\bibliography{database.bib}
\end{document}